\documentclass[12pt]{article}

\pdfoutput=1

\usepackage[a4paper]{geometry}
\usepackage{amsthm}
\usepackage{amsmath}
\usepackage{amsfonts}
\usepackage{amssymb}
\usepackage{graphicx}
\usepackage{enumerate}

\usepackage[colorlinks=true,citecolor=black,linkcolor=black,urlcolor=blue]{hyperref}
\newcommand{\doi}[1]{\href{http://dx.doi.org/#1}{\texttt{doi:#1}}}
\newcommand{\arxiv}[1]{\href{http://arxiv.org/abs/#1}{\texttt{arXiv:#1}}}

\usepackage[usenames]{color}

\theoremstyle{plain}
\newtheorem{theorem}{Theorem}
\newtheorem{lemma}[theorem]{Lemma}
\newtheorem{corollary}[theorem]{Corollary}
\newtheorem{proposition}[theorem]{Proposition}

\theoremstyle{definition}
\newtheorem{definition}[theorem]{Definition}

\theoremstyle{remark}
\newtheorem{remark}[theorem]{Remark}

\renewcommand{\emptyset}{\varnothing}

\begin{document}

\title{A Characterization of Circle Graphs in Terms of Multimatroid Representations}
\author{Robert Brijder\\Hasselt University\\Belgium\\{\small {\texttt{robert.brijder@uhasselt.be}}}
\and Lorenzo Traldi\\Lafayette College\\Easton, Pennsylvania 18042, USA\\{\small {\texttt{traldil@lafayette.edu}}}}
\date{}
\maketitle

\begin{abstract}
The isotropic matroid $M[IAS(G)]$ of a looped simple graph $G$ is a binary matroid equivalent to the isotropic system of $G$. In general, $M[IAS(G)]$ is not regular, so it cannot be represented over fields of characteristic $\neq 2$. The ground set of $M[IAS(G)]$ is denoted $W(G)$; it is partitioned into 3-element subsets corresponding to the vertices of $G$. When the rank function of $M[IAS(G)]$ is restricted to subtransversals of this partition, the resulting structure is a multimatroid denoted $\mathcal{Z}_{3}(G)$. In this paper we prove that $G$ is a circle graph if and only if for every field $\mathbb{F}$, there is an $\mathbb{F}$-representable matroid with ground set $W(G)$, which defines $\mathcal{Z}_{3}(G)$ by restriction. We connect this characterization with several other circle graph characterizations that have appeared in the literature.

\bigskip


\bigskip


\end{abstract}

\section{Introduction}

In this paper a graph may have loops or parallel edges. A graph is~\emph{simple} if it has neither loops nor parallels, and a~\emph{looped simple} graph has no parallels. An edge consists of two distinct half-edges, each incident on one vertex; and an edge is directed by distinguishing one of its half-edges as initial. The degree of a vertex is the number of incident half-edges, and a graph is $d$-regular if its vertices are all of degree $d$. We use the term~\emph{circuit} to refer to a sequence $v_1,h'_1,h_2,v_2,h'_2,\dots,h'_{k-1},h_k=h_1,v_k=v_1$ of vertices and half-edges, such that: for each $i<k$, $h_i$ and $h'_i$ are half-edges incident on $v_i$; for each $i<k$, $h'_i$ and $h_{i+1}$ are half-edges of the same edge; and the half-edges $h_1,h'_1,\dots,h_{k-1},h'_{k-1}$ are mutually distinct. We consider two circuits to be the same if they differ only in orientation or starting point. The sets of vertices and edges of a graph $G$ are denoted by $V(G)$ and $E(G)$, respectively. We assume that the reader is familiar with the basic concepts of matroid theory, see, e.g., \cite{O}.

If $G$ is a connected graph whose vertices are all of even degree then $G$ has an~\emph{Euler circuit}, i.e., a circuit which includes every edge of $G$. In general, if the vertices of $G$ are all of even degree then $G$ will have an~\emph{Euler system}, i.e., a set which contains precisely one Euler circuit for each connected component of $G$.

Circle graphs are an interesting class of simple graphs, which have been studied by many authors during the last 50 years, see, e.g., \cite{Bu, Bco, EI, F, G, GeelenPhD, GPTC, N, RR, Sp, Tn, Z}. Circle graphs arise from Euler systems of 4-regular graphs in the following way.

\begin{definition}
\label{interlacement}
Let $F$ be a 4-regular graph, with an Euler system $C$. Distinct vertices $v,w\in V(F)$ are \emph{interlaced} with respect to $C$ if they appear on one of the circuits of $C$ in the order $vwvw$ or $wvwv$.
\end{definition}

\begin{definition}
\label{circle}
Let $F$ be a 4-regular graph, with an Euler system $C$. The \emph{interlacement graph} $\mathcal{I}(C)$ is the simple graph with $V(\mathcal{I}(C))=V(F)$, in which distinct vertices are adjacent if and only if they are interlaced with respect to $C$. A simple graph that can be realized as an interlacement graph is a~\emph{circle graph}.
\end{definition}
The purpose of this paper is to present a new matroidal characterization of circle graphs, which uses the following notion.

In this paper, for finite sets $X$ and $Y$ we consider $X \times Y$ matrices whose rows and columns are not ordered, but are instead indexed by $X$ and $Y$, respectively. With this notation, the \emph{adjacency matrix} $A(G)$ of the looped simple graph $G$ is a $V(G) \times V(G)$-matrix over $GF(2)$; for $u,v \in V(G)$, the entry of $A(G)$ indexed by $(u,v)$ is $1$ if and only if there is an edge between $u$ and $v$. (In particular, loops lead to nonzero diagonal entries).

\begin{definition}
\label{isomat}
Let $G$ be a looped simple graph, with adjacency matrix $A(G)$. Then the~\emph{isotropic matroid} $M[IAS(G)]$ is the matroid represented by the matrix
\[
IAS(G)=\begin{pmatrix}
I & A(G) & I+A(G)
\end{pmatrix}
\]
over $GF(2)$, the field with two elements.
\end{definition}

The $I$ in Definition~\ref{isomat} is the $V(G) \times V(G)$ identity matrix, and the notation $IAS(G)=\begin{pmatrix}
I & A(G) & I+A(G)
\end{pmatrix}$
indicates that for each $v \in V(G)$, the $v$ rows of $I$, $A(G)$, and $I+A(G)$ are concatenated to form the $v$ row of $IAS(G)$. The matroid elements corresponding to the $v$ columns of $I$, $A(G)$, and $I+A(G)$ are denoted $\phi_G(v)$, $\chi_G(v)$ and $\psi_G(v)$, respectively, and the ground set of $M[IAS(G)]$ (i.e., the set containing all these elements corresponding to columns of $IAS(G)$), is denoted $W(G)$. A subset of $W(G)$ that includes exactly one of $\phi_G(v),\chi_G(v),\psi_G(v)$ for each $v\in V(G)$ is a \emph{transversal} of $W(G)$; a subset of a transversal is a \emph{subtransversal} of $W(G)$. The families of subtransversals and transversals of $W(G)$ are denoted $\mathcal{S}(G)$ and $\mathcal{T}(G)$, respectively.

Isotropic matroids were introduced in~\cite{Tnewnew}, and have been studied further in~\cite{BT1, BT2, BT3, Tra}. One of the simplest properties of isotropic matroids is this: if $G$ has a connected component with more than two vertices, then $M[IAS(G)]$ is not a regular matroid~\cite{Tnewnew}. Despite this simple property, it turns out that circle graphs are characterized by a special kind of regularity associated with isotropic matroids. This special kind of regularity involves a substructure of $M[IAS(G)]$.

\begin{definition}
\label{threematroid}
If $G$ is a looped simple graph then the~\emph{isotropic 3-matroid} $\mathcal{Z}_3(G)$ is the $3$-tuple $(W(G),\Omega,r)$, where $\Omega = \{ \{\phi_G(v),\chi_G(v),\psi_G(v)\} \mid v \in V(G) \}$ and $r$ is the restriction of
the rank function of $M[IAS(G)]$ to $\mathcal{S}(G)$.
\end{definition}

It follows from \cite[Proposition~41]{Tnewnew} that the isotropic 3-matroid of a graph is a \emph{multimatroid}, a notion introduced by Bouchet~\cite{B1, B2, B3, B4}. Like ordinary matroids, multimatroids are uniquely determined by their bases, circuits, or independent sets. An \emph{independent set} of $\mathcal{Z}_3(G)$ is a subtransversal $I \in \mathcal{S}(G)$ with $r(I) = |I|$. A \emph{circuit} of $\mathcal{Z}_3(G)$ is a subtransversal $C \in \mathcal{S}(G)$ that is minimal (w.r.t.\ inclusion) with the property that it is not an independent set. A \emph{basis} of $\mathcal{Z}_3(G)$ is an independent set of $\mathcal{Z}_3(G)$ that is maximal (w.r.t.\ inclusion) with this property. It follows from \cite[Proposition~5.5]{B1} that all bases of $\mathcal{Z}_3(G)$ have a common cardinality equal to $|\Omega|=|V(G)|$. Consequently, if $X$ is the set of bases (independent sets, circuits, resp.) of $M[IAS(G)]$, then $X \cap \mathcal{S}(G)$ is the set of bases (independent sets, circuits, resp.) of $\mathcal{Z}_3(G)$.

\begin{definition}
Let $\mathbb{F}$ be a field. Then $\mathcal{Z}_3(G)$ is \emph{representable} over $\mathbb{F}$ if there is an $\mathbb{F}$-matrix $B$ with columns indexed by $W(G)$, such that the rank function of $\mathcal{Z}_3(G)$ agrees with the rank function of $B$ when restricted to $\mathcal{S}(G)$.
\end{definition}

To say the same thing in a different way: $\mathcal{Z}_3(G)$ is representable over~$\mathbb{F}$ if and only if there is an~$\mathbb{F}$-representable matroid $M$ on $W(G)$, whose rank function agrees with the rank function of $M[IAS(G)]$ when restricted to $\mathcal{S}(G)$. Such a matroid $M$ is said to~\emph{shelter} $\mathcal{Z}_3(G)$. Notice that $\mathcal{Z}_3(G)$ is $GF(2)$-representable by definition; it is sheltered by $M[IAS(G)]$. Our main result is that representability over other fields characterizes circle graphs.

\begin{theorem}
\label{main}
Let $G$ be a simple graph. Then any one of the following conditions implies the others.

\begin{enumerate}

\item $G$ is a circle graph.

\item The 3-matroid $\mathcal{Z}_{3}(G)$ is representable over every field.

\item The 3-matroid $\mathcal{Z}_{3}(G)$ is representable over some field of characteristic $\neq 2$.

\end{enumerate}
\end{theorem}

Theorem~\ref{main} shows that the theory of isotropic 3-matroids is quite different from the more familiar theory of graphic matroids: all graphic matroids are representable over all fields, but the only isotropic 3-matroids representable over all fields are those that come from circle graphs.

Here is an outline of the paper. In Section~\ref{sec:reps} we provide some details about sheltering matroids for $\mathcal{Z}_3(G)$. In Section~\ref{sec:proof1} we deduce the implication $3 \Rightarrow 1$ of Theorem~\ref{main} from the results of Section~\ref{sec:reps} and Bouchet's circle graph obstructions theorem~\cite{Bco}. In Section~\ref{sec:circpart} we summarize the signed interlacement machinery of~\cite{Tsign}, which associates matrices over $GF(2)$ and $\mathbb{R}$ with circuit partitions in 4-regular graphs. In Section~\ref{sec:easy} we use this machinery to show that if $G$ is a circle graph, then $\mathcal{Z}_3(G)$ is representable over~$\mathbb{R}$. This argument is fairly direct, and suffices to prove the implication $1 \Rightarrow 3$ of Theorem~\ref{main}. Getting Condition~2 into the picture is more difficult, because the matrix machinery of~\cite{Tsign} fails over fields with $\textrm{char}(\mathbb{F}) > 2$. In Section~\ref{sec:based} we develop a special case of the signed interlacement machinery, which works over all fields. In Section~\ref{sec:proof2} we complete the proof of Theorem~\ref{main}, and in Section~\ref{sec:example} we detail the constructions used in the proof for a small example. In Section~\ref{sec:naji} we discuss the connection between Theorem~\ref{main} and Naji's characterizations of circle graphs \cite{N1, N}. Finally, in Section~\ref{sec:other} we formulate a more detailed form of Theorem~\ref{main} in terms of multimatroids, and we use that to characterize matroid planarity.

\section{Representations of sheltering matroids}
\label{sec:reps}

We begin by recalling the definition of local equivalence.

\begin{definition} Let $G$ be a looped simple graph and $v$ a vertex of $G$.
\begin{itemize}
\item The graph obtained from $G$ by complementing the loop status of $v$ is denoted $G_{\ell}^{v}$.
\item The graph obtained from $G$ by complementing the adjacency status of every pair of neighbors of $v$ is the~\emph{simple local complement} of $G$ at $v$, denoted $G_{s}^{v}$.
\item The graph obtained from $G$ by complementing the adjacency status of every pair of neighbors of $v$ and the loop status of every neighbor of $G$ is
 the~\emph{non-simple local complement} of $G$ at $v$, denoted $G_{ns}^{v}$.
\item The equivalence relation on looped simple graphs generated by loop complementations and local complementations is~\emph{local equivalence}.
\end{itemize}
\end{definition}

For a graph $G$ and $X \subseteq V(X)$, we denote the subgraph of $G$ induced by $V(G) \setminus X$ by $G-X$.
\begin{definition}
Let $G$ and $H$ be looped simple graphs. Then $H$ is a~\emph{vertex-minor} of $G$ if there is a graph $G'$ that is locally equivalent to $G$, such that $H=G'-X$ for some subset $X \subseteq V(G')$.
\end{definition}

In analogy with Definition~\ref{isomat}, if $G$ is a looped simple graph with adjacency matrix $A(G)$, then we define the~\emph{restricted} isotropic matroid $M[IA(G)]$ to be the matroid represented by the matrix
$ IA(G)=\begin{pmatrix}
I & A(G)
\end{pmatrix} $
over $GF(2)$. That is, $M[IA(G)]$ is the submatroid of $M[IAS(G)]$ that includes only the $\phi$ and $\chi$ elements. The~\emph{isotropic 2-matroid} $\mathcal{Z}_{2}(G)$ is the $3$-tuple $(U,\Omega,r)$, where $\Omega = \{\{\phi_G(v),\chi_G(v)\} \mid v \in V(G)\}$, $U = \bigcup \Omega$, and $r$ is the restriction of the rank function of $M[IA(G)]$ to subtransversals $S \in \mathcal{S}(G)$ that involve only $\phi$ and $\chi$ elements. A \emph{sheltering matroid} for $\mathcal{Z}_{2}(G)$ is a matroid with the same ground set $U$, whose rank function restricts to the rank function of $\mathcal{Z}_{2}(G)$. If $\mathbb{F}$ is a field and $\mathcal{Z}_{2}(G)$ has an $\mathbb{F}$-representable sheltering matroid, then $\mathcal{Z}_{2}(G)$ is \emph{$\mathbb{F}$-representable}.

If $\mathcal{Z}_{2}(G)$ or $\mathcal{Z}_{3}(G)$ has an $\mathbb{F}$-representable sheltering matroid of rank $|V(G)|$, then $\mathcal{Z}_{2}(G)$ or $\mathcal{Z}_{3}(G)$ is \emph{strictly} $\mathbb{F}$-representable. Both $\mathcal{Z}_{2}(G)$ and $\mathcal{Z}_{3}(G)$ are strictly $GF(2)$-representable by definition.

\begin{lemma}
\label{twostrict}
If $G$ is a simple graph then $\mathcal{Z}_{2}(G)$ is $\mathbb{F}$-representable if and only if $\mathcal{Z}_{2}(G)$ is strictly $\mathbb{F}$-representable.
\end{lemma}

\begin{proof}
For each $v \in V(G)$, the transversal $\{\chi_{G}(v)\}\cup\{\phi_{G}(w)\mid v\neq w \in V(G)\}$ is dependent in $M[IAS(G)]$. If $M$ is a sheltering matroid for $\mathcal{Z}_{2}(G)$, then this transversal must be dependent in $M$ too. Hence the $\phi_G$ elements span $M$, so the rank $r(M)$ is no more than the number of $\phi_G$ elements, which is $|V(G)|$. On the other hand, $\{ \phi_G(v) \mid v \in V(G)\}$ is an independent transversal of $M[IAS(G)]$, so it is an independent set of $M$; consequently $r(M) \geq |V(G)|$.
\end{proof}

\begin{lemma}
\label{obvious}If $\mathcal{Z}_{3}(G)$ is (strictly) $\mathbb{F}$-representable then $\mathcal{Z}_{2}(G)$ is (strictly) $\mathbb{F}$-representable.
\end{lemma}

\begin{proof}
If $M$ is a sheltering matroid for $\mathcal{Z}_{3}(G)$ then the submatroid $N$ consisting of elements of $M$ that correspond to elements of $\mathcal{Z}_{2}(G)$ is a sheltering matroid. If $M$ is a strict sheltering matroid then $\Phi(G)=\{\phi_G(v) \mid v \in V(G)\}$ is a basis of $M$. As $N$ contains $\Phi(G)$, $\Phi(G)$ is also a basis of $N$; hence $r(N)=|V(G)|$.
\end{proof}

\begin{lemma}
\label{localeq}If $G$ and $H$ are locally equivalent up to isomorphism, then $\mathcal{Z}_{3}(G)$ is (strictly) $\mathbb{F}$-representable if and only if $\mathcal{Z}_{3}(H)$ is (strictly) $\mathbb{F}$-representable.
\end{lemma}

\begin{proof}
The 3-matroids $\mathcal{Z}_{3}(G)$ and $\mathcal{Z}_{3}(H)$ are isomorphic, see \cite{Tnewnew}.
\end{proof}

\begin{lemma}
\label{vminor}If $\mathcal{Z}_{3}(G)$ is (strictly) $\mathbb{F}$-representable and $H$ is a vertex-minor of $G$, then $\mathcal{Z}_{3}(H)$ is (strictly) $\mathbb{F}$-representable.
\end{lemma}

\begin{proof}
Suppose $M$ is a matroid that shelters $\mathcal{Z}_{3}(G)$. If $v\in V(G)$, then $M^{\prime}=(M/\phi_{G}(v))-\chi_{G}(v)-\psi_{G}(v)$ is a matroid that shelters $\mathcal{Z}_{3}(G-v)$. Moreover, the rank $r(M^{\prime})$ is no more than $r(M)-1$, because $\phi_{G}(v)$ has been contracted, and no less than $|V(G)|-1$, because $\{\phi_{G}(w) \mid v\neq w \in V(G) \}$ is an independent set of $\mathcal{Z}_{3}(G-v)$. Consequently if $r(M)=|V(G)|$, then $r(M')=|V(G)|-1=|V(G-v)|$.

The lemma follows from these observations and Lemma \ref{localeq}.
\end{proof}

We remark that Lemma~\ref{vminor} does not hold for $\mathcal{Z}_{2}(G)$: a graph $G$ such that $\mathcal{Z}_{2}(G)$ is strictly $\mathbb{F}$-representable may have a vertex-minor $H$ such that $\mathcal{Z}_{2}(H)$ is not $\mathbb{F}$-representable. Some examples are given at the end of Section~\ref{sec:proof1}.

\begin{lemma}
\label{standardtwo}

Suppose $G$ is a simple graph and $\mathcal{Z}_{2}(G)$ is $\mathbb{F}$-representable. Then there is a $V(G) \times V(G)$ matrix $A$ with entries in $\mathbb{F}$, with these properties.

\begin{enumerate}

\item $A$ has nonzero entries in precisely the same places where the adjacency matrix $A(G)$ has nonzero entries.

\item
$\begin{pmatrix}
I & A\\
\end{pmatrix}$,
with $I$ an identity matrix, represents a matroid that shelters $\mathcal{Z}_{2}(G)$, with the columns of $I$ and $A$ corresponding to the $\phi_G$ and $\chi_G$ elements, respectively.
\end{enumerate}
\end{lemma}

\begin{proof}
As $\mathcal{Z}_{2}(G)$ is $\mathbb{F}$-representable, for some $m$ there is an $m\times 2n$ matrix $Q$ with entries in $\mathbb{F}$, which represents a matroid that shelters $\mathcal{Z}_{2}(G)$. We presume the columns of $Q$ are ordered with the $\phi$ columns first, and then the $\chi$ columns. The $\phi_G$ elements of $M[IAS(G)]$ are independent, so the first $n$ columns of $Q$ are linearly independent. It follows that elementary row operations can be used to bring $Q$ into the form
\[
Q'=
\begin{pmatrix}
I & A \\
0 & A'
\end{pmatrix}
\text{.}
\]
Elementary row operations have no effect on the matroid represented by a
matrix, so $Q'$ represents a matroid that shelters $\mathcal{Z}_{2}(G)$.

If any entry of $A'$ is not $0$, then the corresponding column of $Q'$ is not included in the span of the columns of $I$. Consequently if $v\in V(G)$ is the vertex whose $\chi_{G}(v)$ element corresponds to this column of $Q'$, then $\{\chi_{G}(v)\}\cup\{\phi_{G}(w)\mid w\neq v\}$ is an independent set of the matroid represented by $Q'$. This set is a transversal of $W(G)$, so it is also an independent set of $M[IAS(G)]$. This is incorrect, however; $\{\chi_{G}(v)\}\cup\{\phi_{G}(w)\mid w\neq v\}$ is dependent in $M[IAS(G)]$, because the $v$ entry of the $\chi_{G}(v)$ column of $IAS(G)$ is $0$. We conclude that $A'=0$.

The fact that $\{\chi_{G}(v)\}\cup\{\phi_{G}(w)\mid w\neq v\}$ is a dependent set of $M[IAS(G)]$ also implies that the $v$ entry of the $\chi_{G}(v)$ column of $Q'$ is $0$. That is, every diagonal entry of $A$ is $0$.

Now, suppose $v\neq w\in V(G)$. The subtransversal $\{\chi_{G}(w)\}\cup
\{\phi_{G}(x)\mid x\notin\{v,w\}\}$ of $W(G)$ is independent in $M[IAS(G)]$ if and only if $v$ and $w$ are neighbors in $G$. By hypothesis, this subtransversal is independent in $M[IAS(G)]$ if and only the corresponding columns of $Q'$ are linearly independent. As the columns of $Q'$ corresponding to $\phi_{G}$ elements are columns of the identity matrix, and the $w$ entry of the $\chi_{G}(w)$ column of $Q'$ is $0$, it follows that the $v$ entry of this column of $Q'$ is $0$ if and only if the $v$ entry of the corresponding column of $IAS(G)$ is $0$.
\end{proof}

\begin{lemma}
\label{standardthree}

Let $G$ be a simple graph with $n$ vertices, such that $\mathcal{Z}_{3}(G)$ is $\mathbb{F}$-representable. Then for some $m\geq n$, there is an $m\times3n$ matrix
\[
P=
\begin{pmatrix}
I & A & B\\
0 & 0 & C
\end{pmatrix}
\]
with entries in $\mathbb{F}$, which satisfies the following.

\begin{enumerate}

\item $P$ represents a matroid that shelters $\mathcal{Z}_{3}(G)$, with the columns of $I$, $A$ and $B$ corresponding to the $\phi_G$, $\chi_G$ and $\psi_G$ elements, respectively.

\item The submatrix $I$ is an $n\times n$ identity matrix.

\item The submatrix $A$ has nonzero entries in precisely the same places where the adjacency matrix $A(G)$ has nonzero entries.

\end{enumerate}

If $\mathcal{Z}_{3}(G)$ is strictly $\mathbb{F}$-representable, then there is a matrix $P$ which moreover satisfies the three conditions below.

\begin{enumerate}
\setcounter{enumi}{3}
\item $C=0$.

\item The diagonal entries of the submatrix $B$ are all equal to $1$.

\item If $v$ and $w$ are neighbors in $G$, then $B_{vw} B_{wv} = 1$.

\end{enumerate}

\end{lemma}

\begin{proof}Let $Q$ be an $m \times 3n$ matrix with entries in $\mathbb{F}$, which represents a matroid that shelters $\mathcal{Z}_{3}(G)$. We presume the columns of $Q$ are ordered with the $\phi$ columns first, then the $\chi$ columns, and then the $\psi$ columns. The $\phi_G$ elements of $M[IAS(G)]$ are independent, so the first $n$ columns of $Q$ are linearly independent. Elementary row operations can be used to bring $Q$ into the form
\[
P=
\begin{pmatrix}
I & A & B\\
0 & A' & C
\end{pmatrix}
\text{.}
\]
Elementary row operations have no effect on the matroid represented by a
matrix, so $P$ inherits property 1 from $Q$.

The proof of Lemma~\ref{standardtwo} shows that $A'=0$ and $A$ satisfies property 3.

If $\mathcal{Z}_{3}(G)$ is strictly $\mathbb{F}$-representable, we may start with a matrix $Q$ of rank $n$. Then $P$ also is of rank $n$, so $A'=C=0$.

If $v \in V(G)$ then $\{\phi_G(w) \mid v \neq w \} \cup \{\psi_G(v)\}$ is a transversal of $W(G)$, and is independent in $M[IAS(G)]$. It follows that the corresponding columns of $P$ are linearly independent; this requires that $B_{vv}\neq 0$. We may multiply the $v$ column of $B$ by $1/B_{vv}$, without affecting the represented matroid. For simplicity we still use $B$ and $P$ to denote the matrices resulting from these column multiplications.

It remains to verify property 6. If $v$ and $w$ are neighbors in $G$ then $\{\phi_G(x) \mid x \notin \{v,w\}\} \cup \{\psi_G(v),\psi_G(w)\}$ is a transversal of $W(G)$, which is dependent in $M[IAS(G)]$. It follows that the corresponding columns of $P$ are linearly dependent, so
\[
\begin{pmatrix}
1 & B_{vw}\\
B_{wv} & 1
\end{pmatrix}
\]
is a singular matrix.
\end{proof}

Notice that property 3 of Lemma~\ref{standardthree} specifies the locations of nonzero entries in the submatrix $A$. We do not have such specific information about $B$, however. Properties 5 and 6 guarantee nonzero entries on the diagonal, and in locations that correspond to edges of $G$; but there may be nonzero entries in other places, and these locations may vary from one sheltering matroid to another.

\section{Bouchet's obstructions}
\label{sec:proof1}

In this section we prove the implication $3 \Rightarrow 1$ of Theorem~\ref{main}: if $\mathcal{Z}_3(G)$ is $\mathbb{F}$-representable over some field with $\textrm{char}(\mathbb{F}) \neq 2$, then $G$ is a circle graph. The crucial ingredient of the proof is the following well-known result.

\begin{theorem}
\label{Bthm}
(Bouchet's circle graph obstructions theorem~\cite{Bco}) A simple graph is not a circle graph if and only if it has one of the graphs pictured in Figure~\ref{circmf4} as a vertex-minor.
\end{theorem}

Bouchet's theorem is useful in proving Theorem~\ref{main} because of the following result. This result is closely related to a statement given (without proof) in \cite[page 36]{GeelenPhD} in the context of delta-matroids.

\begin{proposition}
\label{obstruct}
If $\mathbb{F}$ is a field with $\textrm{char}(\mathbb{F}) \neq 2$ and $G \in \{W_{5},BW_{3},W_{7}\}$, then $\mathcal{Z}_{2}(G)$ is not representable over $\mathbb{F}$.
\end{proposition}

\begin{proof}
There is a transversal $T$ of $W(BW_3)$ which includes only $\phi$ and $\chi$ elements, such that the restriction of $M[IAS(BW_3)]$ to $T$ is isomorphic to the Fano matroid $F_{7}$. The Fano matroid is not representable over $\mathbb{F}$, so no $\mathbb{F}$-representable matroid can shelter $\mathcal{Z}_{2}(BW_{3})$.

The proposition is a bit harder to verify for $W_{5}$. To
establish notation, we set $V(W_{5})=\{1,2,3,4,5,6\}$ and $E(W_{5})=\{1i,(i-1)i\mid 2 \leq i \leq 6\}\cup\{26\}$. Suppose $A$ satisfies Lemma \ref{standardtwo} with $G=W_{5}$. If we multiply a column
of $A$ by a nonzero element of $\mathbb{F}$, then we do not change the matroid represented by
$\begin{pmatrix}
I & A\\
\end{pmatrix}$.
Consequently we may presume that in each column of $A$, the first nonzero entry is $1$.  Property 1 of Lemma \ref{standardtwo} now tells us that
\[
A=
\begin{pmatrix}
0 & 1 & 1 & 1 & 1 & 1\\
a & 0 & b & 0 & 0 & b^{\prime}\\
c & d & 0 & d^{\prime} & 0 & 0\\
e & 0 & f & 0 & f^{\prime} & 0\\
g & 0 & 0 & h & 0 & h^{\prime}\\
i & j & 0 & 0 & j^{\prime} & 0
\end{pmatrix}
\text{,}
\]
where the entries represented by letters are nonzero elements of $\mathbb{F}$.

The submatrix of $IAS(W_{5})$ corresponding to the subtransversal $\{\phi_{W_{5}}(4), \phi_{W_{5}}(5), \linebreak \chi_{W_{5}}(3),\chi_{W_{5}}(6)\}$ is represented by the rank 3 matrix
\[
\begin{pmatrix}
0 & 0 & 1 & 1\\
0 & 0 & 1 & 1\\
0 & 0 & 0 & 0\\
1 & 0 & 1 & 0\\
0 & 1 & 0 & 1\\
0 & 0 & 0 & 0
\end{pmatrix}
\text{.}
\]
Property 2 of Lemma~\ref{standardtwo} tells us that the corresponding submatrix of
$\begin{pmatrix}
I & A \\
\end{pmatrix}
$,
\[
\begin{pmatrix}
0 & 0 & 1 & 1\\
0 & 0 & b & b^{\prime}\\
0 & 0 & 0 & 0\\
1 & 0 & f & 0\\
0 & 1 & 0 & h^{\prime}\\
0 & 0 & 0 & 0
\end{pmatrix}
\text{,}
\]
is also of rank $3$. We deduce that $b=b^{\prime}$. Similar arguments tell us
that $d=d^{\prime}$, $f=f^{\prime}$, $h=h^{\prime}$ and $j=j^{\prime}$.

\begin{figure}
[ptb]
\begin{center}
\includegraphics[scale=0.8]{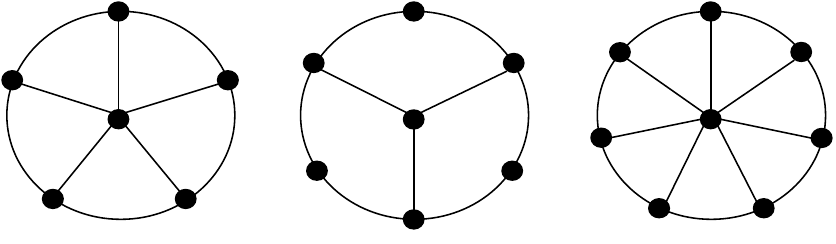}
\caption{$W_{5}$, $BW_{3}$ and $W_{7}$.}
\label{circmf4}\end{center}
\end{figure}

Now, consider the transversal $T=\{\phi_{W_{5}}(1),\chi_{W_{5}}(2),\chi_{W_{5}}(3),\chi_{W_{5}}(4),\chi_{W_{5}}(5),$ $\chi_{W_{5}}(6)\}$ of $W(W_{5})$. The corresponding submatrix of $\begin{pmatrix}
I & A \\
\end{pmatrix}
$ is
\[
\begin{pmatrix}
1 & 1 & 1 & 1 & 1 & 1\\
0 & 0 & b & 0 & 0 & b\\
0 & d & 0 & d & 0 & 0\\
0 & 0 & f & 0 & f & 0\\
0 & 0 & 0 & h & 0 & h\\
0 & j & 0 & 0 & j & 0
\end{pmatrix}
\text{.}
\]
A direct calculation shows that the determinant of this matrix is $2bdfhj$, which is nonzero in $\mathbb{F}$ but $0$ in $GF(2)$. It follows that the transversal $T$ is independent in
$M[\begin{pmatrix}
I & A \\
\end{pmatrix}
]$
 and dependent in $M[IAS(G)]$; this contradicts property 2 of Lemma~\ref{standardtwo}.

The proposition may be verified for $W_{7}$ by a closely analogous argument.
\end{proof}

We deduce the contrapositive of the implication $3 \Rightarrow 1$ of Theorem~\ref{main}.

\begin{corollary}
\label{onetwo}
If $\mathcal{Z}_{3}(G)$ is representable over some field $\mathbb{F}$ with $\textrm{char}(\mathbb{F}) \neq 2$, then $G$ is a circle graph.
\end{corollary}

\begin{proof}
Suppose $\mathcal{Z}_{3}(G)$ is representable over a field $\mathbb{F}$ with $\textrm{char}(\mathbb{F}) \neq 2$. Lemmas~\ref{obvious} and~\ref{vminor} tell us that for every vertex-minor $H$ of $G$, $\mathcal{Z}_{2}(H)$ is also representable over $\mathbb{F}$. According to Proposition~\ref{obstruct}, it follows that no vertex-minor of $G$ is isomorphic to $W_5,BW_3$ or $W_7$. Theorem~\ref{Bthm} now tells us that $G$ is a circle graph.
\end{proof}

Before proceeding we take a moment to observe that in general, Lemma~\ref{vminor} and Corollary~\ref{onetwo} do not hold for~$\mathcal{Z}_{2}(G)$. Let $G$ be a bipartite graph which is a fundamental graph for the cycle or cocycle matroid of some nonplanar graph. Then $M[IA(G)]$ is isomorphic to the direct sum $M \oplus M^*$, where $M$ is the cycle matroid of the nonplanar graph. It follows that $M \oplus M^*$ is a strict sheltering matroid for~$\mathcal{Z}_{2}(G)$. It is well known that graphic and cographic matroids are representable over all fields~\cite[Lemma 5.1.3 and Corollary 5.1.6]{O}, so~$M \oplus M^*$ is representable over all fields. Hence $\mathcal{Z}_{2}(G)$ is strictly representable over all fields. However a theorem of deFraysseix~\cite{F} asserts that $G$ cannot be a circle graph.

Two such examples were discussed in \cite[Section 8]{BT2}. They are pictured in Figure~\ref{fourfig4}; $G_1$ is a fundamental graph for $M(K_{5})$ and $G_2$ is a fundamental graph for $M(K_{3,3})$. It is not hard to directly confirm deFraysseix's assertion that neither $G_1$ nor $G_2$ is a circle graph~\cite{F}; $G_1$ has $BW_3$ as a vertex-minor, and $G_2$ has $W_5$ as a vertex-minor.

\begin{figure}[ptb]\centering
\includegraphics[scale=0.8]{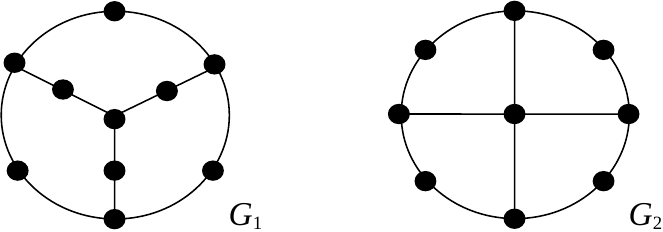}
\caption{Fundamental graphs of $M(K_{5})$ and $M(K_{3,3})$.}
\label{fourfig4}\end{figure}

\section{Circuit partitions in 4-regular graphs}
\label{sec:circpart}
In this section we summarize some ideas and results from~\cite{Tsign}; we refer to that paper for a more detailed discussion.

If $F$ is a 4-regular graph, then, at each vertex $v \in V(F)$, there are three different~\emph{transitions}, i.e., partitions of the four incident half-edges into two pairs. We use $\mathfrak{T}(F)$ to denote the set of transitions in $F$, and we refer to one pair of half-edges incident at $v$ as a~\emph{single transition}. If $C$ is an Euler system of $F$, then $C$ can be used to label the elements of~$\mathfrak{T}(F)$ in the following way. First, orient each circuit of $C$. Then the transition at $v$ that pairs together half-edges which appear consecutively on the incident circuit of $C$ is denoted $\phi_C(v)$; the other transition that is consistent with the edge directions defined by the incident circuit of $C$ is denoted $\chi_C(v)$; and the third transition, which is inconsistent with the edge directions defined by the incident circuit of $C$, is denoted $\psi_C(v)$. These transition labels are not changed if the orientations of some circuits of $C$ are reversed.

The reappearance of the $\phi,\chi,\psi$ symbols used to label elements of isotropic matroids is no coincidence. If $C$ is an Euler system of $F$, then there is a bijection $W(\mathcal{I}(C)) \leftrightarrow \mathfrak{T}(F)$ given by $\phi_{\mathcal{I}(C)}(v) \leftrightarrow \phi_C(v)$, $\chi_{\mathcal{I}(C)}(v) \leftrightarrow \chi_C(v)$ and $\psi_{\mathcal{I}(C)}(v) \leftrightarrow \psi_C(v)$ for all $v \in V(F)$. This bijection relates each transversal $T \in \mathcal{T}(\mathcal{I}(C))$ to a partition of $E(F)$ into edge-disjoint circuits, and it turns out that the 3-matroid $\mathcal{Z}_3(\mathcal{I}(C))$ is determined by these partitions. Before giving details we should emphasize that according to the definition given in the introduction, for us a ``circuit'' is a closed trail. In particular a circuit in a 4-regular graph must not visit any half-edge more than once, but it may visit a vertex twice.

\begin{definition}
\label{circpartition}
Let $F$ be a 4-regular graph. A~\emph{circuit partition} of $F$ is a partition of $E(F)$ into edge-disjoint circuits.
\end{definition}

A circuit partition $P$ is determined by choosing one transition $P(v)$ at each vertex $v$ of $F$. There are three transitions at each vertex, so there are $3^{|V(F)|}$ circuit partitions.

\begin{definition}
\label{touch}
Let $P$ be a circuit partition in a 4-regular graph $F$. Then the~\emph{touch-graph} $Tch(P)$ is the graph with a vertex $v_{\gamma}$ for each $\gamma \in P$ and an edge $e_v$ for each $v \in V(F)$, such that $e_v$ is incident on $v_{\gamma}$ in $Tch(P)$ if and only if $\gamma$ is incident on $v$ in $F$.
\end{definition}

Examples of Definition~\ref{touch} appear in Figure~\ref{fourfig3a1}. On the left we see two circuit partitions $P$ and $P'$ in a 4-regular graph $F$. To follow a circuit of $P$ or $P'$ walk along the edges of $F$, making sure to maintain the line status (plain, heavy, or dashed) when traversing a vertex. (The line status may change in the middle of an edge.) On the right we see $Tch(P)$ and $Tch(P')$.

\begin{figure}[tb]
\centering
\includegraphics[scale=0.8]{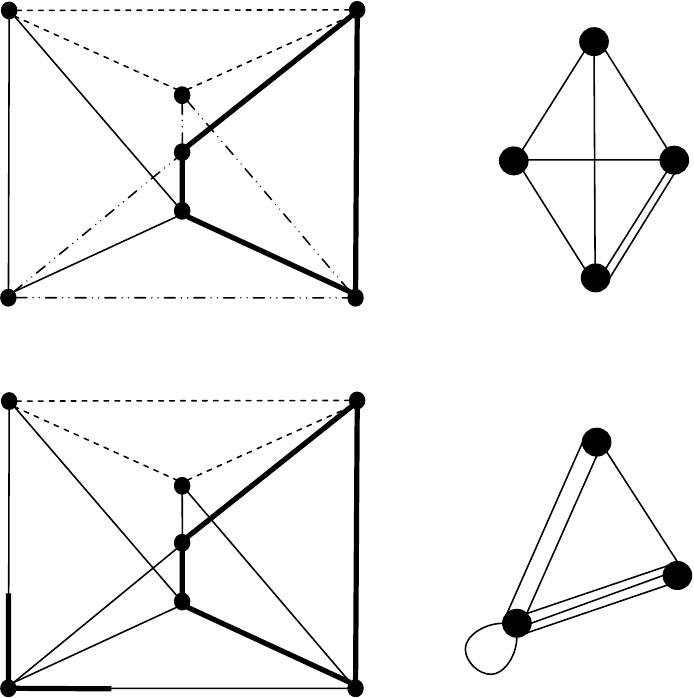}
\caption{Two circuit partitions and their touch-graphs.}
\label{fourfig3a1}
\end{figure}

If $P$ is a circuit partition in $F$, then every half-edge $h$ in $F$ has a ``shadow'' half-edge $\pi_P(h)$ in $Tch(P)$, defined in the following way. If $h$ is incident on a vertex $v$, $\gamma$ is the circuit of $P$ that includes $h$, and $\{h,h'\}$ is the single transition of $P$ that includes $h$, then $\pi_P(h)=\pi_P(h')$ is a half-edge of $Tch(P)$ that is contained in $e_v$ and incident on $v_{\gamma}$. For simplicity we use the notation $\pi_P(h)=\overline{h}$ when $P$ is clear from the context.

Also, every walk $W=v_1,e_1,v_2,\ldots,e_{k-1},v_k$ in $F$ has a ``shadow'' $\pi_P(W)=\overline{W}$, which is a walk in $Tch(P)$. If $\gamma_1$ is the circuit of $P$ that includes $e_1$, then $v_{\gamma_1}$ is the first vertex of $\overline{W}$. As we follow $W$ in $F$, each time we pass through a vertex $v_i$ we traverse two half-edges incident on $v_i$; say $h$ before $v_i$, and $h'$ after $v_i$. If the transition determined by $\{h,h'\}$ is not a transition of $P$, then the edge $\{\overline{h},\overline{h'}\}$ of $Tch(P)$ is added to $\overline{W}$. If the transition determined by $\{h,h'\}$ is a transition of $P$, then no edge is added to $\overline{W}$. (In this instance we are walking along a circuit of $P$ as we pass through $v$ on $W$, so the ``shadow'' is standing still on a vertex of $Tch(P)$.)

If $W$ is a closed walk in $F$ then $\overline{W}$ is a closed walk in $Tch(P)$. Of course $\overline{W}$ may be much shorter than $W$; for instance if $W \in P$ then $\overline{W}$ is just $v_W$.

Let $D$ be a directed version of $Tch(P)$. Let $\mathbb{F}$ be a field, and for each directed walk $W$ in $Tch(P)$ let $z_D(W)$ be the vector in $\mathbb{F}^{E(Tch(P))}$ obtained by tallying $+1$ in the $e$ coordinate each time $W$ traverses $e$ in accordance with the $D$ direction, and $-1$ in the $e$ coordinate each time $W$ traverses $e$ against the $D$ direction. Then the subspace of $\mathbb{F}^{E(Tch(P))}$ spanned by $\{z_D(W) \mid W \text{ is a closed walk in }Tch(P) \}$ is the $\emph{cycle space}$ of $Tch(P)$ over $\mathbb{F}$, denoted $Z_D(Tch(P))$. We refer to standard texts in graph theory, like~\cite{B}, for detailed discussions of cycle spaces.

\begin{definition}
\label{funcirc}
Let $F$ be a 4-regular graph with an Euler system $C$. For each $v\in V(F)$, there are two circuits obtained by following a circuit of $C$ from $v$ to $v$. These are the~\emph{induced circuits} of $C$ at $v$.
\end{definition}

\begin{theorem}(\cite[Section 2]{Tsign})
\label{space}
Let $C$ be an Euler system of a 4-regular graph $F$, and let $\Gamma$ be a set that includes one induced circuit of $C$ at each $v \in V(F)$, along with an orientation for that circuit. Then for every circuit partition $P$ of $F$ and every directed version $D$ of $Tch(P)$, $\{z_D(\overline{\gamma}) \mid \gamma \in \Gamma \}$ spans $Z_D(Tch(P))$ over both $GF(2)$ and $\mathbb{R}$.
\end{theorem}

We think of Theorem~\ref{space} as a surprising result because the touch-graphs of circuit partitions in $F$ are quite varied. There are touch-graphs in which all edges are loops (the touch-graphs of the smallest circuit partitions, the Euler systems), touch-graphs in which no edges are loops (the touch-graphs of the maximal circuit partitions), and many other touch-graphs between these extremes. Despite this variation, Theorem~\ref{space} allows us to describe spanning sets in the cycle spaces of all touch-graphs in a consistent way.

When citing Theorem~\ref{space}, we use the notation $\Gamma=\{ C_{\Gamma}(v) \mid v \in V(F) \}$.
Let $M_{\mathbb{R},\Gamma}(C,P,D)$ be the $V(F) \times V(F)$ matrix whose $v$ row vector is obtained from $z_D(\overline{C_{\Gamma}(v)})$ by relabelling according to the bijection $e_w \leftrightarrow w$. Then the cycle space $Z_D(Tch(P))$ over $\mathbb{R}$ corresponds to the row space of $M_{\mathbb{R},\Gamma}(C,P,D)$. Over $GF(2)$, $Z_D(Tch(P))$ corresponds to the row space of the matrix $M(C,P)$ obtained from $M_{\mathbb{R},\Gamma}(C,P,D)$ by reducing all entries modulo 2.
We can use the simple $M(C,P)$ notation when we work over $GF(2)$ because the value of $M_{\mathbb{R},\Gamma}(C,P,D)_{vw}$ modulo 2 is independent of both $D$ and $\Gamma$.

Theorem~\ref{space} implies that the matroid defined by $M_{\mathbb{R},\Gamma}(C,P,D)$ over $\mathbb{R}$ is the same as the matroid defined by $M(C,P)$ over $GF(2)$:

\begin{corollary} \label{cor:matroid_GF2_R}
Let $S$ be a subset of $V(F)$. Let $M_{\mathbb{R}}(S)$ be the submatrix of $M_{\mathbb{R},\Gamma}(C,P,D)$ consisting of columns corresponding to elements of $S$, and let $M(S)$ be the matrix obtained from $M_{\mathbb{R}}(S)$ by reducing its entries modulo $2$. Then the rank of $M(S)$ over $GF(2)$ is the same as the rank of $M_{\mathbb{R}}(S)$ over $\mathbb{R}$.
\end{corollary}
\begin{proof}
Suppose that $S$ is minimal among subsets of $V(F)$ for which the rank of $M(S)$ over $GF(2)$ is not the same as the rank of $M_{\mathbb{R}}(S)$ over $\mathbb{R}$. If the rank of $M(S)$ over $GF(2)$ is strictly larger than the rank of $M_{\mathbb{R}}(S)$ over $\mathbb{R}$, then the columns of $M(S)$ are linearly independent over $GF(2)$ but the columns of $M_{\mathbb{R}}(S)$ are linearly dependent over $\mathbb{R}$, and hence also over $\mathbb{Q}$. That is, there are rational numbers $r_s$, $s \in S$, not all $0$, such that if we multiply the $s$ column of $M_{\mathbb{R}}(S)$ by $r_s$ for each $s \in S$, then the sum of the resulting column vectors is $0$. Eliminating common factors and multiplying by denominators, we may presume that the numbers $r_s$ are relatively prime integers. Then not all of the $r_s$ are divisible by $2$, so they define a linear dependence of the columns of $M(S)$ over $GF(2)$, contradicting the hypothesis that the columns of $M(S)$ are independent over $GF(2)$. We conclude that the rank of $M(S)$ over $GF(2)$ is strictly smaller than the rank of $M_{\mathbb{R}}(S)$ over $\mathbb{R}$.

Thus the minimality of $S$ guarantees that the columns of $M_{\mathbb{R}}(S)$ are linearly independent over $\mathbb{R}$, but linearly dependent when their entries are reduced modulo $2$, and no proper subset of the columns of $M(S)$ is dependent. Let $\kappa(S) \in GF(2)^{V(F)}$ be the vector whose $v$ entry is $1$ if and only if $v \in S$. Then $\kappa(S)$ is an element of the orthogonal complement of the row space of $M(C,P)$ over $GF(2)$, so Theorem~\ref{space} tells us that $\{e_s \mid s\in S\}$ is a cocycle of $Tch(P)$. That is, there is a proper subset $P_0$ of $P$ such that $S$ is the set of vertices of $F$ incident on both a circuit from $P_0$ and a circuit from $P \setminus P_0$. Let $\kappa_{\mathbb{R}}(S)$ be the vector in $\mathbb{Z}^{V(F)}$ whose $v$ coordinate is nonzero if and only if $v \in S$, with the $v$ coordinate equal to $1$ if $e_v$ is directed in $D$ from a circuit in $P_0$ to a circuit in $P \setminus P_0$, and the $v$ coordinate equal to $-1$ if $e_v$ is directed in $D$ from a circuit in $P \setminus P_0$ to a circuit in $P_0$. Then $\kappa_{\mathbb{R}}(S)$ is a cocycle of $D$, so $\kappa_{\mathbb{R}}(S)$ is an element of the orthogonal complement of $Z_{D}(Tch(P))$ over $\mathbb{R}$. But according to Theorem~\ref{space}, this contradicts the hypothesis that the columns of $M_{\mathbb{R}}(S)$ are linearly independent over $\mathbb{R}$.
\end{proof}

It is not difficult to describe the matrix $M_{\mathbb{R},\Gamma}(C,P,D)$ explicitly. Given a vertex $v$, the $v$ row of $M_{\mathbb{R},\Gamma}(C,P,D)$ is obtained by following $C_{\Gamma}(v)$ in $F$, and tallying contributions to the ``shadow'' of $C_{\Gamma}(v)$ in $Tch(P)$. At a vertex $w \neq v$, the $\phi_C(w)$ transition of $F$ appears in $C_{\Gamma}(v)$, so the contribution of a passage of $C_{\Gamma}(v)$ through $w$ is $0$ if $P(w)=\phi_C(w)$, and $\pm 1$ if $P(w) \neq \phi_C(w)$. Considering that $C_{\Gamma}(v)$ may pass through $w$ twice, we see that $M_{\mathbb{R},\Gamma}(C,P,D)_{vw} \in \{-2,-1,0,1,2 \}$. At $v$ itself, the $\chi_C(v)$ transition of $F$ appears in $C_{\Gamma}(v)$, so the contribution of a passage of $C_{\Gamma}(v)$ through $v$ is $0$ if $P(v)=\chi_C(v)$, and $\pm 1$ if $P(v) \neq \chi_C(v)$. Considering that $C_{\Gamma}(v)$ passes through $v$ once, we see that $M_{\mathbb{R},\Gamma}(C,P,D)_{vv} \in \{-1,0,1 \}$.

Notice that if $\Gamma$ is changed by reversing the orientation of $C_{\Gamma}(v)$, then the $v$ row of $M_{\mathbb{R},\Gamma}(C,P,D)$ is multiplied by $-1$. Also, if $D$ is changed by reversing the direction of $e_w$, then the $w$ column of $M_{\mathbb{R},\Gamma}(C,P,D)$ is multiplied by $-1$.

The purpose of the subindex $\mathbb{R}$ in $M_{\mathbb{R},\Gamma}(C,P,D)$ is to remind us that in general, Theorem~\ref{space} fails over fields $\mathbb{F}$ with $\textrm{char}(\mathbb{F}) > 2$. As an example, consider the 4-regular graph $F$ with $V(F)=\{a,b,c\}$, which has two edges connecting each pair of vertices. We index the edges of $F$ as $e_1, \ldots, e_6$ in such a way that an Euler circuit $C$ is $a,e_1,b,e_2,c,e_3,a,e_4,b,e_5,c,e_6$. Let $P$ be the circuit partition given by the transitions $\chi_C(a),\chi_C(b)$, and $\psi_C(c)$. Then $P$ has only one element, the Euler circuit $a,e_1,b,e_5,c,e_2,b,$ $e_4,a,e_6,c,e_3$. Choose a directed version $D$ of $Tch(P)$ so that the initial half-edge of $e_a$ involves the single transition $\{e_1,e_3\}$, the initial half-edge of $e_b$ involves the single transition $\{e_1,e_5\}$ and the initial half-edge of $e_c$ involves the single transition $\{e_2,e_5\}$. If $\Gamma=\{e_1e_2e_3,e_5e_6e_1,e_3e_4e_5\}$ then
\[
\det M_{\mathbb{R},\Gamma}(C,P,D)=
\det \begin{pmatrix}
0 & 1 & 1 \\
-1 & 0 & 1 \\
1 & -1 & 1
\end{pmatrix}=3.
\]
The fact that the determinant of $M_{\mathbb{R},\Gamma}(C,P,D)$ is 0 (mod 3) indicates that the row space of $M_{\mathbb{R},\Gamma}(C,P,D)$ is a proper subspace of $Z_D(Tch(P))$ over $GF(3)$.

A weak version of Theorem~\ref{space} does hold over all fields:

\begin{proposition}
\label{arbfield}
Let $C$ be an Euler system of a 4-regular graph $F$, let $\Gamma$ be a set of oriented induced circuits of $C$, and let $\mathbb{F}$ be a field. Then for every circuit partition $P$ of $F$ and every directed version $D$ of $Tch(P)$, $\{z_D(\overline{\gamma}) \mid \gamma \in \Gamma \}$ is contained in the cycle space of $Tch(P)$ over $\mathbb{F}$.
\end{proposition}
\begin{proof}
If $\gamma \in \Gamma$, then $\overline{\gamma}$ is a closed walk in $Tch(P)$, so $z_D(\overline{\gamma})$ is an element of $Z_D(Tch(P))$.
\end{proof}

\section{Theorem~\ref{main} over \texorpdfstring{$\mathbb{R}$}{R}}
\label{sec:easy}

In this section we show that if $F$ is a 4-regular graph with an Euler system $C$, then the machinery of Section~\ref{sec:circpart} provides an $\mathbb{R}$-representable sheltering matroid for $\mathcal{Z}_3(\mathcal{I}(C))$. The basic idea is to construct a single matrix which contains $M_{\mathbb{R},\Gamma}(C,P,D)$ matrices for all circuit partitions of $F$. In order to do this we need a systematic way to choose oriented versions of touch-graphs. The approach we use is not the only possible one, but it is convenient because it is easy to describe and it is connected with signed interlacement systems that have been discussed by several authors~\cite{Bu, Jo, Lau, MP, Tn, Tsign}.

Let $C$ be an oriented Euler system of a 4-regular graph $F$, i.e., each circuit of $C$ is given with an orientation. Let $\Gamma^{o}$ be a set of~\emph{consistently oriented} induced circuits of $C$, i.e., each $C_{\Gamma^{o}}(v) \in \Gamma^{o}$ is oriented consistently with the circuit of $C$ that contains it. For each $v \in V(F)$ we designate one passage of $C$ through $v$ as $v^-$ and the other passage of $C$ through $v$ as $v^+$, in such a way that when we follow a circuit of $C$ in accord with its given orientation, we have $\ldots, v^-, C_{\Gamma^{o}}(v), v^+, \ldots$.

\begin{definition}
\label{bigmatrix}We use these vertex signs to define an integer matrix
\[
IAS_{\Gamma^{o}}(C)=
\begin{pmatrix}
I & A & B
\end{pmatrix}
\]
as follows.
\begin{enumerate}
\item $I$ is the $V(F) \times V(F)$ identity matrix.

\item $A$ is the $V(F) \times V(F)$ matrix given by:
\begin{gather*}
A_{vw}=
\begin{cases}
1\text{,} & \text{if }v\neq w \text{ and }C_{\Gamma^{o}}(v)\text{ includes
}w^+ \text{ but not } w^-\\
-1\text{,} & \text{if }v\neq w \text{ and }C_{\Gamma^{o}}(v)\text{ includes
}w^- \text{ but not } w^+\\
0\text{,} & \text{otherwise}
\end{cases}
\end{gather*}

\item $B$  is the $V(F) \times V(F)$ matrix given by:
\begin{gather*}
B_{vw}=
\begin{cases}
0\text{,} & \text{if }v\neq w \text{ and }C_{\Gamma^{o}}(v)\text{ includes
neither }w^+ \text{ nor } w^-\\
2\text{,} & \text{if }v\neq w \text{ and }C_{\Gamma^{o}}(v)\text{ includes both
}w^+ \text{ and } w^-\\
1\text{,} & \text{otherwise}
\end{cases}
\end{gather*}

\end{enumerate}

\end{definition}

The columns of $IAS_{\Gamma^{o}}(C)$ are indexed by elements of $\mathfrak{T}(F)$ in the following way: for each $w\in V(F)$ the $w$ column of $I$ corresponds to $\phi_C(w)$, the $w$ column of $A$ corresponds to $\chi_C(w)$, and the $w$ column of $B$ corresponds to $\psi_C(w)$.
\begin{definition}
\label{standardg}
Let $C$ be an oriented Euler system of $F$, $\Gamma^{o}$ a set of consistently oriented induced circuits of $C$, and $P$ a circuit partition of $F$. Then $D_{\Gamma^{o}}$ denotes the directed version of $Tch(P)$ in which for each $v \in V(F)$, the initial half-edge of $e_v$ is $\overline{h}$, where $h$ is the half-edge of $F$ directed toward $v$ on $C_{\Gamma^{o}}(v)$.
\end{definition}

\begin{proposition}
\label{standardmg}
Let $C$ be an oriented Euler system of $F$, $\Gamma^{o}$ a set of consistently oriented induced circuits of $C$, and $P$ a circuit partition of $F$. Then $M_{\mathbb{R},\Gamma^{o}}(C,P,D_{\Gamma^{o}})$ is identical to the $V(F) \times V(F)$ submatrix of $IAS_{\Gamma^{o}}(C)$ whose columns correspond to the transitions in $P$.
\end{proposition}

\begin{proof}
Let $w \in V(F)$. We index the four half-edges incident on $w$ as $h_1,h_2,h_3,h_4$ so that a circuit of $C$ is $ \ldots ,h_1,w^{-},h_2, \ldots ,h_3,w^{+},h_4, \ldots $ and $C_{\Gamma^{o}}(w)$ is $w^{-},h_2, \ldots ,h_3,w^{+}$. Then $\pi_P(h_3)=\overline{h_3}$ is the initial half-edge of $e_w$ in $D_{\Gamma^{o}}$.

Suppose $P(w)=\phi_C(w)$. For $v \neq w$, a passage of $C_{\Gamma^{o}}(v)$ through $w$ is $h_1,w,h_2$ or $h_3,w,h_4$. Either way, $C_{\Gamma^{o}}(v)$ follows a circuit of $P$ through $w$, so $M_{\mathbb{R},\Gamma^{o}}(C,P,D_{\Gamma^{o}})_{vw}=0$. On the other hand, the passage of $C_{\Gamma^{o}}(w)$ through $w$ is $h_3,w,h_2$. This passage agrees with the direction of $e_w$ in $D_{\Gamma^{o}}$, so $M_{\mathbb{R},\Gamma^{o}}(C,P,D_{\Gamma^{o}})_{ww}=1$. We see that the $w$ column of the identity matrix is the same as the $w$ column of $M_{\mathbb{R},\Gamma^{o}}(C,P,D_{\Gamma^{o}})$.

If $P(w)=\chi_C(w)$ then $h_3,w,h_2$ and $h_1,w,h_4$ are the passages of circuit(s) of $P$ through $w$. The initial half-edge of $e_w$ in $D_{\Gamma^{o}}$ is $\overline{h_3}=\overline{h_2}$, and the terminal half-edge is $\overline{h_1}=\overline{h_4}$. The passage of $C_{\Gamma^{o}}(w)$ through $w$ is $h_3,w,h_2$, so $M_{\mathbb{R},\Gamma^{o}}(C,P,D_{\Gamma^{o}})_{ww}=0$. Suppose $v\neq w \in V(F)$. If $C_{\Gamma^{o}}(v)$ includes the passage $h_1,w,h_2$ through $w$ then as this passage does not agree with the direction of $e_w$ in $D_{\Gamma^{o}}$, it contributes $-1$ to $M_{\mathbb{R},\Gamma^{o}}(C,P,D_{\Gamma^{o}})_{vw}$. On the other hand, if $C_{\Gamma^{o}}(v)$ includes the passage $h_3,w,h_4$ through $w$ then this passage contributes $1$ to $M_{\mathbb{R},\Gamma^{o}}(C,P,D_{\Gamma^{o}})_{vw}$. We see that the $w$ column of $A$ is the same as the $w$ column of $M_{\mathbb{R},\Gamma^{o}}(C,P,D_{\Gamma^{o}})$.

If $P(w)=\psi_C(w)$ then $h_1,w,h_3$ and $h_2,w,h_4$ are the passages of circuit(s) of $P$ through $w$. The initial half-edge of $e_w$ in $D_{\Gamma^{o}}$ is $\overline{h_3}=\overline{h_1}$, and the terminal half-edge is $\overline{h_2}=\overline{h_4}$. The passage of $C_{\Gamma^{o}}(w)$ through $w$ is $h_3,w,h_2$, so $M_{\mathbb{R},\Gamma^{o}}(C,P,D_{\Gamma^{o}})_{ww}=1$. Suppose $v\neq w \in V(F)$. If $C_{\Gamma^{o}}(v)$ includes the passage $h_1,w,h_2$ through $w$, then this passage contributes $1$ to $M_{\mathbb{R},\Gamma^{o}}(C,P,D_{\Gamma^{o}})_{vw}$. If $C_{\Gamma^{o}}(v)$ includes the passage $h_3,w,h_4$ through $w$, then this passage also contributes $1$ to $M_{\mathbb{R},\Gamma^{o}}(C,P,D_{\Gamma^{o}})_{vw}$. We see that the $w$ column of $B$ is the same as the $w$ column of $M_{\mathbb{R},\Gamma^{o}}(C,P,D_{\Gamma^{o}})$.
\end{proof}
\begin{corollary}
\label{bigoneg}
Let $C$ be an oriented Euler system of $F$, and $\Gamma^{o}$ a set of consistently oriented induced circuits of $C$. Let $M_{\mathbb{R}}[IAS_{\Gamma^{o}}(C)]$ be the matroid represented over $\mathbb{R}$ by $IAS_{\Gamma^{o}}(C)$. Then $M_{\mathbb{R}}[IAS_{\Gamma^{o}}(C)]$ is a strict $3$-sheltering matroid for $\mathcal{Z}_3(\mathcal{I}(C))$.
\end{corollary}
\begin{proof}
If $T$ is a transversal of $W(G)$ then $F$ has a circuit partition $P$, determined by the transitions corresponding to elements of $T$. The submatrix of $IAS_{\Gamma^{o}}(C)$ corresponding to $T$ is formed by using the columns of $IAS_{\Gamma^{o}}(C)$ corresponding to these transitions, so Proposition~\ref{standardmg} tells us that this submatrix is $M_{\mathbb{R},\Gamma^{o}}(C,P,D_{\Gamma^{o}})$. Corollary~\ref{cor:matroid_GF2_R} tells us that the rank of each set $S \subseteq T$ of columns of this matrix is the same over $GF(2)$ and $\mathbb{R}$.

It is clear from Definition~\ref{bigmatrix} that when we reduce $IAS_{\Gamma^{o}}(C)$ modulo 2, the resulting matrix is $IAS(\mathcal{I}(C))$; hence the last sentence of the preceding paragraph tells us that $M_{\mathbb{R}}[IAS_{\Gamma^{o}}(C)]$ is a $3$-sheltering matroid for $\mathcal{Z}_3(\mathcal{I}(C))$. It is also clear from Definition~\ref{bigmatrix} that the rank of $IAS_{\Gamma^{o}}(C)$ is $n$, so $M_{\mathbb{R}}[IAS_{\Gamma^{o}}(C)]$ is a strict sheltering matroid.
\end{proof}

Corollary~\ref{bigoneg} gives us the implication $1 \Rightarrow 3$ of Theorem~\ref{main}. The implication $3 \Rightarrow 1$ was verified in Section~\ref{sec:proof1}, so we have demonstrated that conditions 1 and 3 are equivalent. The implication $2 \Rightarrow 3$ is obvious, but $1 \Rightarrow 2$ is difficult; as noted at the end of Section~\ref{sec:circpart}, Theorem~\ref{space} does not hold over fields of odd characteristic, because the determinant of $M_{\mathbb{R},\Gamma^{o}}(C,P,D)$ may be an odd integer $\notin \{-1,1\}$. (Theorem~\ref{space} does guarantee that the determinant of $M_{\mathbb{R},\Gamma^{o}}(C,P,D)$ cannot be an even integer other than 0.) The key to our proof of Theorem~\ref{main} is the fact that for certain special sets of induced circuits, $\det M_{\mathbb{R},\Gamma^{o}}(C,P,D) \in \{-1,0,1\}$ is always true. We verify this fact in Section~\ref{sec:based}.

\section{Based sets of induced circuits}
\label{sec:based}

Let $F$ be a 4-regular graph, with an Euler system $C$ and an edge $e$. Definition~\ref{funcirc} implies that for each vertex $v$ in the same connected component of $F$ as $e$, one induced circuit of $C$ at $v$ includes $e$, and the other excludes $e$.

\begin{definition}
\label{basedcirc}
Let $E$ be a set that contains one edge from each component of $F$, and $C$ an Euler system of $F$. A set of oriented induced circuits of $C$ that exclude every element of $E$ is said to be \emph{based on} $E$. If $\Gamma$ is based on $E$ then we also say that the elements of $E$ are~\emph{base edges} for $\Gamma$.
\end{definition}

We use the notation $\Gamma_E$ to indicate that $\Gamma$ is based on $E$. Notice that if we are given $C$ and $E$ then there are $2^{|E|}$ different sets denoted $\Gamma_E$, distinguished by the orientations of the induced circuits.

The usefulness of based induced circuits stems from the following.

\begin{lemma}
\label{sameops}
Let $E$ be a set of base edges in a 4-regular graph $F$, let $C$ and $\widetilde{C}$ be Euler systems of $F$, and let $\Gamma_E$ and $\widetilde{\Gamma}_E$ be based sets of oriented induced circuits of $C$ and $\widetilde{C}$ (respectively). Then these two properties hold.
\begin{enumerate}[I.]

\item For every circuit partition $P$ and every directed version $D$ of $Tch(P)$, the matrices $M_{\mathbb{R},\Gamma_E}(C,P,D)$ and $M_{\mathbb{R},\widetilde{\Gamma}_E}(\widetilde{C},P,D)$ are row equivalent over $\mathbb{Z}$. That is, one matrix can be obtained from the other using the operations ``multiply a row by $-1$'' and ``add an integer multiple of one row to a different row''.

\item There is a single sequence of row operations which transforms $M_{\mathbb{R},\Gamma_E}(C,P,D)$ into $M_{\mathbb{R},\widetilde{\Gamma}_E}(\widetilde{C},P,D)$ for every choice of $P$ and $D$.
\end{enumerate}
\end{lemma}

\begin{proof}
According to a theorem of Kotzig~\cite{K}, it suffices to verify properties I and II when $C$ and $\widetilde{C}$ are related through a single $\kappa$-transformation. That is, we may assume that $\widetilde{C}$ is the Euler system of $F$ obtained from $C$ by reversing the direction in which the circuit of $C$ incident at $v$ traverses $C_{\Gamma_E}(v)$. It is easy to see that this partial reversal affects some transition labels at vertices that appear precisely once on $C_{\Gamma_E}(v)$: $\phi_C(v)=\psi_{\widetilde{C}}(v)$, $\psi_C(v)=\phi_{\widetilde{C}}(v)$, and if $w \neq v$ appears precisely once on $C_{\Gamma_E}(v)$ then $\chi_C(w)=\psi_{\widetilde{C}}(w)$ and $\psi_C(w)=\chi_{\widetilde{C}}(w)$. Otherwise, transition labels with respect to $C$ and $\widetilde{C}$ are the same.

In order to avoid proliferation of subcases we assume that the Euler circuits included in $C$ are given with orientations, the circuits of $\widetilde{C}$ inherit orientations from the circuits of $C$, and all the induced circuits included in $\Gamma_E$ and $\widetilde{\Gamma}_E$ respect these orientations. As mentioned before the example at the end of Section~\ref{sec:circpart}, this presumption does not involve a significant loss of generality because the effect of reversing the orientation of an induced circuit is simply to multiply the corresponding row of $M_{\mathbb{R},\Gamma_E}(C,P,D)$ or $M_{\mathbb{R},\widetilde{\Gamma}_E}(\widetilde{C},P,D)$ by $-1$.

Property I. Suppose first that $v$ is a vertex with $P(v)=\psi_C(v)$.

1. Suppose $w \neq v \in V(F)$ and $C_{\Gamma_E}(w)$ does not intersect $C_{\Gamma_E}(v)$. Then $C_{\Gamma_E}(w)$ and $\widetilde{C}_{\widetilde{\Gamma}_E}(w)$ are the same circuit in $F$, so they provide the same $w$ row for $M_{\mathbb{R},\Gamma_E}(C,P,D)$ and $M_{\mathbb{R},\widetilde{\Gamma}_E}(\widetilde{C},P,D)$.

2. Suppose $w \in V(F)$ and $C_{\Gamma_E}(w)$ is contained in $C_{\Gamma_E}(v)$. (It may be that $v=w$.) Then $C_{\Gamma_E}(w)$ and $\widetilde{C}_{\widetilde{\Gamma}_E}(w)$ are the same circuit in $F$, but oriented in opposite directions.  It follows that the $w$ row of $M_{\mathbb{R},\Gamma_E}(C,P,D)$ is the negative of the $w$ row of $M_{\mathbb{R},\widetilde{\Gamma}_E}(\widetilde{C},P,D)$.

3. Suppose $w \neq v \in V(F)$ and $C_{\Gamma_E}(w)$ contains $C_{\Gamma_E}(v)$.  Then the vectors $z_D(\overline{C_{\Gamma_E}(w)})$ and $z_D(\overline{\widetilde{C}_{\widetilde{\Gamma}_E}(w)})$ are obtained using the same contributions from passages through vertices outside $C_{\Gamma_E}(v)$, and opposite contributions from passages through vertices inside $C_{\Gamma_E}(v)$, other than $v$ itself. It follows that if $x \neq v \in V(F)$ then the $x$ coordinates of $z_D(\overline{\widetilde{C}_{\widetilde{\Gamma}_E}(w)})$ and $z_D(\overline{C_{\Gamma_E}(w)})-2 \cdot z_D(\overline{C_{\Gamma_E}(v)})$ are the same.

In contrast, the situation at $v$ is complicated by the fact that the passages of $C_{\Gamma_E}(v)$ and $C_{\Gamma_E}(w)$ through $v$ are different. We index the half-edges incident on $v$ as $h_1,h_2,h_3,h_4$ in such a way that the circuit of $C$ incident on $v$ is $e,\ldots,h_1,v,h_2,\ldots,h_3,v,h_4,\ldots$, where $e \in E$. Then the fact that $P(v)=\psi_C(v)$ indicates that the single transitions at $v$ included in $P$ are $\{h_1,h_3\}$ and $\{h_2,h_4\}$. It follows that the passages $h_1,v,h_2$; $h_1,v,h_4$; $h_3,v,h_2$ and $h_3,v,h_4$ all have the same ``shadow'' in $Tch(P)$. Let us suppose for convenience that the direction of $e_v$ in $D$ follows the common shadow of these four passages through $v$. (If the opposite is true then all the $v$ coordinates mentioned in the next paragraph should be multiplied by $-1$; but the conclusion of the paragraph after that is unchanged.)

The passages of $C_{\Gamma_E}(w)$ through $v$ are $h_1,v,h_2$ and $h_3,v,h_4$, so the $v$ coordinate of $z_D(\overline{C_{\Gamma_E}(w)})$ is 2. The passage of $C_{\Gamma_E}(v)$ through $v$ is $h_3,v,h_2$, so the $v$ coordinate of $z_D(\overline{C_{\Gamma_E}(v)})$ is 1. It follows that the $v$ coordinate of $z_D(\overline{C_{\Gamma_E}(w)})-2 \cdot z_D(\overline{C_{\Gamma_E}(v)})$ is 0. As the passages of $\widetilde{C}_{\widetilde{\Gamma}_E}(w)$ through $v$ are $h_1,v,h_3$ and $h_2,v,h_4$, the $v$ coordinate of $z_D(\overline{\widetilde{C}_{\widetilde{\Gamma}_E}(w)})$ is also 0.

We see that $z_D(\overline{\widetilde{C}_{\widetilde{\Gamma}_E}(w)})=z_D(\overline{C_{\Gamma_E}(w)})-2 \cdot z_D(\overline{C_{\Gamma_E}(v)})$. That is, the $w$ row of $M_{\mathbb{R},\widetilde{\Gamma}_E}(\widetilde{C},P,D)$ is obtained by subtracting 2 times the $v$ row of $M_{\mathbb{R},\Gamma_E}(C,P,D)$ from the $w$ row of $M_{\mathbb{R},\Gamma_E}(C,P,D)$.

4. Suppose now that $C_{\Gamma_E}(w)$ intersects $C_{\Gamma_E}(v)$ but neither induced circuit contains the other. Then $\widetilde{C}_{\widetilde{\Gamma}_E}(w)$ is obtained from $C_{\Gamma_E}(w)$ by replacing $C_{\Gamma_E}(v) \cap C_{\Gamma_E}(w)$ with $C_{\Gamma_E}(v) \setminus C_{\Gamma_E}(w)$, oriented in the opposite direction from its orientation in $C_{\Gamma_E}(v)$. Consequently $z_D(\overline{\widetilde{C}_{\widetilde{\Gamma}_E}(w)})$ $=z_D(\overline{C_{\Gamma_E}(w)})-z_D(\overline{C_{\Gamma_E}(v)})$, as far as the $x$ coordinates with $x \notin \{v,w\}$ are concerned.

For $v$ and $w$, though, we have complications similar to the complications at $v$ in part 3. As discussed in part 3, the passages of $C_{\Gamma_E}(v)$ and $C_{\Gamma_E}(w)$ through $v$ have the same shadow in $Tch(P)$, so the $v$ coordinate of $z_D(\overline{C_{\Gamma_E}(w)})-z_D(\overline{C_{\Gamma_E}(v)})$ is 0. The $v$ coordinate of $z_D(\overline{\widetilde{C}_{\widetilde{\Gamma}_E}(w)})$ is also 0, because $P(v)=\phi_{\widetilde{C}}(v)$.

We claim that the $w$ coordinates of $z_D(\overline{C_{\Gamma_E}(w)})-z_D(\overline{C_{\Gamma_E}(v)})$ and $z_D(\overline{\widetilde{C}_{\widetilde{\Gamma}_E}(w)})$ are the same, too. To verify this claim we need a detailed analysis of half-edges like the one in part 3. Let $h_1,h_2,h_3,h_4$ be the half-edges of $F$ incident on $w$, indexed in such a way that a circuit of $C$ is of the form $e,\ldots,h_1,w,h_2,\ldots,h_3,w,h_4,\ldots$, with $e \in E$.

If $P(w)=\phi_C(w)$ then the single transitions at $w$ included in $P$ are $\{h_1,h_2\}$ and $\{h_3,h_4\}$. It follows that the passages $h_1,w,h_3$; $h_1,w,h_4$; $h_2,w,h_3$ and $h_2,w,h_4$ all have the same ``shadow'' in $Tch(P)$. Let us assume that the direction of $e_w$ in $D$ follows this shadow. The passage of $C_{\Gamma_E}(w)$ through $w$ is $h_3,w,h_2$, so the $w$ coordinate of $z_D(\overline{C_{\Gamma_E}(w)})$ is $-1$. The passage of $C_{\Gamma_E}(v)$ through $w$ might be $h_1,w,h_2$ or $h_3,w,h_4$; $\{h_1,h_2\}$ and $\{h_3,h_4\}$ are single transitions of $P$, so the $w$ coordinate of $z_D(\overline{C_{\Gamma_E}(v)})$ is 0. The passage of $\widetilde{C}_{\widetilde{\Gamma}_E}(w)$ through $w$ is $h_3,w,h_1$ or $h_4,w,h_2$. Either way, the $w$ coordinate of $z_D(\overline{\widetilde{C}_{\widetilde{\Gamma}_E}(w)})$ is $-1$, the same as the $w$ coordinate of $z_D(\overline{C_{\Gamma_E}(w)})-z_D(\overline{C_{\Gamma_E}(v)})$.

If $P(w)=\chi_C(w)$ then the single transitions at $w$ included in $P$ are $\{h_1,h_4\}$ and $\{h_2,h_3\}$, so the passages $h_1,w,h_2$; $h_1,w,h_3$; $h_4,w,h_2$ and $h_4,w,h_3$ all have the same ``shadow'' in $Tch(P)$. Let us assume that the direction of $e_w$ in $D$ follows this shadow. The passage of $C_{\Gamma_E}(w)$ through $w$ is $h_3,w,h_2$; as $\{h_2,h_3\}$ is a single transition of $P$, the $w$ coordinate of $z_D(\overline{C_{\Gamma_E}(w)})$ is 0. The passage of $C_{\Gamma_E}(v)$ through $w$ might be $h_1,w,h_2$ or $h_3,w,h_4$; we refer to these possibilities as subcases (i) and (ii), respectively. The passage of $\widetilde{C}_{\widetilde{\Gamma}_E}(w)$ through $w$ is $h_3,w,h_1$ in subcase (i) and $h_4,w,h_2$ in subcase (ii); in either subcase the shadows in $Tch(P)$ of the passages of $C_{\Gamma_E}(v)$ and $\widetilde{C}_{\widetilde{\Gamma}_E}(w)$ through $w$ are opposites of each other. Consequently the $w$ coordinates of $z_D(\overline{C_{\Gamma_E}(w)})-z_D(\overline{C_{\Gamma_E}(v)})$ and $z_D(\overline{\widetilde{C}_{\widetilde{\Gamma}_E}(w)})$ are the same if $P(w)=\chi_C(w)$.

The last possibility to consider is $P(w)=\psi_C(w)$. Then the single transitions at $w$ included in $P$ are $\{h_1,h_3\}$ and $\{h_2,h_4\}$, and the passages $h_1,w,h_2$; $h_1,w,h_4$; $h_3,w,h_2$ and $h_3,v,h_4$ all have the same ``shadow'' in $Tch(P)$. We presume that the direction of $e_w$ in $D$ follows this shadow. The passage of $C_{\Gamma_E}(w)$ through $w$ is $h_3,w,h_2$, so the $w$ coordinate of $z_D(\overline{C_{\Gamma_E}(w)})$ is 1. The passage of $C_{\Gamma_E}(v)$ through $w$ is $h_1,w,h_2$ or $h_3,w,h_4$, so the $w$ coordinate of $z_D(\overline{C_{\Gamma_E}(v)})$ is 1. The passage of $\widetilde{C}_{\widetilde{\Gamma}_E}(w)$ through $w$ is $h_3,w,h_1$ or $h_4,w,h_2$, so the $w$ coordinate of $z_D(\overline{\widetilde{C}_{\widetilde{\Gamma}_E}(w)})$ is 0. We see that the $w$ coordinates of $z_D(\overline{C_{\Gamma_E}(w)})-z_D(\overline{C_{\Gamma_E}(v)})$ and $z_D(\overline{\widetilde{C}_{\widetilde{\Gamma}_E}(w)})$ are both 0.

Parts 1--4 complete the proof in case $P(v)=\psi_C(v)$. As noted early in the argument, $P(v)=\psi_C(v)$ implies $P(v)=\phi_{\widetilde{C}}(v)$. In the same way, the case $P(v)=\phi_C(v)$ includes $P(v)=\psi_{\widetilde{C}}(v)$. Consequently the argument for $P(v)=\phi_C(v)$ is obtained by interchanging $C$ and $\widetilde{C}$ in parts 1--4.

It remains to consider the possibility that $P(v)=\chi_C(v)$. The argument follows the same outline as above, but the details are different in some places.

5. If $C_{\Gamma_E}(w)$ does not intersect $C_{\Gamma_E}(v)$ then $C_{\Gamma_E}(w)$ and $\widetilde{C}_{\widetilde{\Gamma}_E}(w)$ are the same circuit in $F$, so $z_D(\overline{C_{\Gamma_E}(w)})=z_D(\overline{\widetilde{C}_{\widetilde{\Gamma}_E}(w)})$.

6. If $C_{\Gamma_E}(w)$ is contained in $C_{\Gamma_E}(v)$ then $C_{\Gamma_E}(w)$ and $\widetilde{C}_{\widetilde{\Gamma}_E}(w)$ are the same circuit in $F$, but with opposite orientations. Hence $z_D(\overline{C_{\Gamma_E}(w)})=-z_D(\overline{\widetilde{C}_{\widetilde{\Gamma}_E}(w)})$.

7. Suppose $v \neq w$ and $C_{\Gamma_E}(w)$ contains $C_{\Gamma_E}(v)$. Then $z_D(\overline{C_{\Gamma_E}(w)})$ and $z_D(\overline{\widetilde{C}_{\widetilde{\Gamma}_E}(w)})$ are obtained using the same contributions from passages through vertices outside $C_{\Gamma_E}(v)$, and opposite contributions from passages through vertices inside $C_{\Gamma_E}(v)$. Hence $z_D(\overline{\widetilde{C}_{\widetilde{\Gamma}_E}(w)})=z_D(\overline{C_{\Gamma_E}(w)})-2 \cdot z_D(\overline{C_{\Gamma_E}(v)})$ as far as passages inside or outside of $C_{\Gamma_E}(v)$ are concerned.

We must still discuss the $v$ coordinate. Again, we index the half-edges incident on $v$ as $h_1,h_2,h_3,h_4$ in such a way that a circuit of $C$ is $e,\ldots,h_1,v,h_2,\ldots,h_3,v,h_4,\ldots$, where $e \in E$. Then the fact that $P(v)=\chi_C(v)$ indicates that the single transitions at $v$ included in $P$ are $\{h_1,h_4\}$ and $\{h_3,h_2\}$. It follows that the passages $h_1,v,h_2$; $h_1,v,h_3$; $h_4,v,h_2$ and $h_4,v,h_3$ all have the same ``shadow'' in $Tch(P)$; we presume that the  direction of $e_v$ in $D$ follows the common shadow of these four passages through $v$. The passages of $C_{\Gamma_E}(w)$ through $v$ are $h_1,v,h_2$ and $h_3,v,h_4$, so the $v$ coordinate of $z_D(\overline{C_{\Gamma_E}(w)})$ is $1-1=0$. The passage of $C_{\Gamma_E}(v)$ through $v$ is $h_3,v,h_2$, so the $v$ coordinate of $z_D(\overline{C_{\Gamma_E}(v)})$ is 0. It follows that the $v$ coordinate of $z_D(\overline{C_{\Gamma_E}(w)})-2 \cdot z_D(\overline{C_{\Gamma_E}(v)})$ is 0. As the passages of $\widetilde{C}_{\widetilde{\Gamma}_E}(w)$ through $v$ are $h_1,v,h_3$ and $h_2,v,h_4$, the $v$ coordinate of $z_D(\overline{\widetilde{C}_{\widetilde{\Gamma}_E}(w)})$ is also $1-1=0$.

Just as before, it follows that $z_D(\overline{\widetilde{C}_{\widetilde{\Gamma}_E}(w)})=z_D(\overline{C_{\Gamma_E}(w)})-2 \cdot z_D(\overline{C_{\Gamma_E}(v)})$.

8. If $C_{\Gamma_E}(w)$ intersects $C_{\Gamma_E}(v)$ but neither contains the other then just as before, $z_D(\overline{\widetilde{C}_{\widetilde{\Gamma}_E}(w)})$ $=z_D(\overline{C_{\Gamma_E}(w)})-z_D(\overline{C_{\Gamma_E}(v)})$ as far as the contributions of vertices $x \notin \{v,w\}$ are concerned.

As in part 7, we suppose that $C$ is of the form $e,\ldots,h_1,v,h_2,\ldots,h_3,v,h_4,\ldots$; the single transitions at $v$ included in $P$ are $\{h_1,h_4\}$ and $\{h_3,h_2\}$; and the direction of $e_v$ in $D$ follows the common ``shadow'' of the passages $h_1,v,h_2$; $h_1,v,h_3$; $h_4,v,h_2$ and $h_4,v,h_3$. The $v$ coordinate of $z_D(\overline{C_{\Gamma_E}(v)})$ is 0, because $P(v)=\chi_{C}(v)$. If $C_{\Gamma_E}(w)$ includes the passage $h_1,v,h_2$ through $v$ then the $v$ coordinate of $z_D(\overline{C_{\Gamma_E}(w)})$ is 1. Also, $\widetilde{C}_{\widetilde{\Gamma}_E}(w)$ includes the passage $h_1,v,h_3$ through $v$, so the $v$ coordinate of $z_D(\overline{\widetilde{C}_{\widetilde{\Gamma}_E}(w)})$ is 1 too. On the other hand, if $C_{\Gamma_E}(w)$ includes the passage $h_3,v,h_4$ through $v$ then the $v$ coordinates of $z_D(\overline{C_{\Gamma_E}(w)})$ and $z_D(\overline{\widetilde{C}_{\widetilde{\Gamma}_E}(w)})$ are both $-1$. We see that the $v$ coordinates of $z_D(\overline{\widetilde{C}_{\widetilde{\Gamma}_E}(w)})$ and  $z_D(\overline{C_{\Gamma_E}(w)})-z_D(\overline{C_{\Gamma_E}(v)})$ are always the same.

We claim that in addition, the $w$ coordinates of $z_D(\overline{\widetilde{C}_{\widetilde{\Gamma}_E}(w)})$ and $z_D(\overline{C_{\Gamma_E}(w)})-z_D(\overline{C_{\Gamma_E}(v)})$ are the same. To verify this claim we can use the same argument as in part 4; the different value of $P(v)$ is irrelevant.

Property II. In the argument above, parts 1 and 5 have $z_D(\overline{\widetilde{C}_{\widetilde{\Gamma}_E}(w)})=z_D(\overline{C_{\Gamma_E}(w)})$, parts 2 and 6 have $z_D(\overline{\widetilde{C}_{\widetilde{\Gamma}_E}(w)})=-z_D(\overline{C_{\Gamma_E}(w)})$, parts 3 and 7 have $z_D(\overline{\widetilde{C}_{\widetilde{\Gamma}_E}(w)})=z_D(\overline{C_{\Gamma_E}(w)})-2 \cdot z_D(\overline{C_{\Gamma_E}(v)})$, and parts 4 and 8 have $z_D(\overline{\widetilde{C}_{\widetilde{\Gamma}_E}(w)})=z_D(\overline{C_{\Gamma_E}(w)})-z_D(\overline{C_{\Gamma_E}(v)})$. This is not quite enough as the case $P(v)=\phi_C(v)$ was not discussed in detail, but instead described by interchanging $C$ and $\widetilde{C}$ in parts 1---4. When we interchange $C$ and $\widetilde{C}$ in the equalities $z_D(\overline{\widetilde{C}_{\widetilde{\Gamma}_E}(w)})=z_D(\overline{C_{\Gamma_E}(w)})$ and $z_D(\overline{\widetilde{C}_{\widetilde{\Gamma}_E}(w)})=-z_D(\overline{C_{\Gamma_E}(w)})$, we obtain the same equalities. When we interchange $C$ and $\widetilde{C}$ in $z_D(\overline{\widetilde{C}_{\widetilde{\Gamma}_E}(w)})=z_D(\overline{C_{\Gamma_E}(w)})-2 \cdot z_D(\overline{C_{\Gamma_E}(v)})$ we obtain $z_D(\overline{C_{\Gamma_E}(w)})=z_D(\overline{\widetilde{C}_{\widetilde{\Gamma}_E}(w)})-2 \cdot z_D(\overline{\widetilde{C}_{\widetilde{\Gamma}_E}(v)})$ or equivalently, $z_D(\overline{C_{\Gamma_E}(w)})+2 \cdot z_D(\overline{\widetilde{C}_{\widetilde{\Gamma}_E}(v)})=z_D(\overline{\widetilde{C}_{\widetilde{\Gamma}_E}(w)})$. This is the same as the original equality because part 2 tells us that $z_D(\overline{\widetilde{C}_{\widetilde{\Gamma}_E}(v)})=-z_D(\overline{C_{\Gamma_E}(v)})$. Similarly, when we interchange $C$ and $\widetilde{C}$ in $z_D(\overline{\widetilde{C}_{\widetilde{\Gamma}_E}(w)})=z_D(\overline{C_{\Gamma_E}(w)})-z_D(\overline{C_{\Gamma_E}(v)})$, we obtain $z_D(\overline{C_{\Gamma_E}(w)})=z_D(\overline{\widetilde{C}_{\widetilde{\Gamma}_E}(w)})-z_D(\overline{\widetilde{C}_{\widetilde{\Gamma}_E}(v)})$, or $z_D(\overline{C_{\Gamma_E}(w)})+z_D(\overline{\widetilde{C}_{\widetilde{\Gamma}_E}(v)})=z_D(\overline{\widetilde{C}_{\widetilde{\Gamma}_E}(w)})$; this is the same as the original equality because $z_D(\overline{\widetilde{C}_{\widetilde{\Gamma}_E}(v)})=-z_D(\overline{C_{\Gamma_E}(v)})$.

It follows that property II holds when $\Gamma_E,\widetilde{\Gamma}_E$ respect the assumption about orientations of induced circuits mentioned at the beginning of the argument above. To deal with induced circuits that might not respect this assumption, we might have to multiply some rows by $-1$; but the same rows will require this multiplication for every $P$. It follows that property II holds when $\widetilde{C}$ is a $\kappa$-transform of $C$, for all choices of $\Gamma_E$ and $\widetilde{\Gamma}_E$.
\end{proof}

\begin{definition}
Let $C$ be an Euler system of a 4-regular graph $F$, and $\Gamma$ a set of oriented induced circuits of $C$. If $v \in V(F)$ then the ``shadow'' of $C_{\Gamma}(v)$ in $Tch(C)$ includes only one edge: $\overline{C_{\Gamma}(v)}=e_v$. We denote by $D_{\Gamma}$ the directed version of $Tch(C)$ in which for each $v \in V(F)$, the direction of $e_v$ is chosen so that when we traverse $C_{\Gamma}(v)$ according to the orientation of $C_{\Gamma}(v)$, our shadow traverses $e_v$ in the $D_{\Gamma}$ direction.
\end{definition}

That is, $D_{\Gamma}$ is the directed version of $Tch(C)$ with $M_{\mathbb{R},\Gamma}(C,C,D_{\Gamma})=I$.

\begin{theorem}
\label{premul}
Suppose $E$ is a set of base edges in a 4-regular graph $F$, $C$ and $\widetilde{C}$ are Euler systems of $F$, and $\Gamma_E$ and $\widetilde{\Gamma}_E$ are based sets of oriented induced circuits of $C$ and $\widetilde{C}$ (respectively).

Then for every circuit partition $P$ of $F$ and every directed version $D$ of $Tch(P)$,
\[
M_{\mathbb{R},\widetilde{\Gamma}_E}(\widetilde{C},P,D)=M_{\mathbb{R},\widetilde{\Gamma}_E}(\widetilde{C},C,D_{\Gamma_E}) \cdot M_{\mathbb{R},\Gamma_E}(C,P,D) \text{.}
\]

Moreover,
\[
M_{\mathbb{R},\widetilde{\Gamma}_E}(\widetilde{C},C,D_{\Gamma_E}) = M_{\mathbb{R},\Gamma_E}(C,\widetilde{C},D_{\widetilde{\Gamma}_E})^{-1} \text{.}
\]
\end{theorem}
\begin{proof}
Let $A$ be the product of elementary matrices corresponding to the row operations of Lemma~\ref{sameops}. Then Lemma~\ref{sameops} tells us that for every circuit partition $P$ and every directed version $D$ of $Tch(P)$, $M_{\mathbb{R},\widetilde{\Gamma}_E}(\widetilde{C},P,D)=A \cdot M_{\mathbb{R},\Gamma_E}(C,P,D)$. In particular, $A=A \cdot I=A \cdot M_{\mathbb{R},\Gamma_E}(C,C,D_{\Gamma_E})=M_{\mathbb{R},\widetilde{\Gamma}_E}(\widetilde{C},C,D_{\Gamma_E})$.

Taking $P=\widetilde{C}$ and $D=D_{\widetilde{\Gamma}_E}$, the equality just proved tells us that
\[
M_{\mathbb{R},\widetilde{\Gamma}_E}(\widetilde{C},C,D_{\Gamma_E}) \cdot M_{\mathbb{R},\Gamma_E}(C,\widetilde{C},D_{\widetilde{\Gamma}_E})=M_{\mathbb{R},\widetilde{\Gamma}_E}(\widetilde{C},\widetilde{C},D_{\widetilde{\Gamma}_E})=I \text{.}
\]
\end{proof}

We refer to the equalities of Theorem~\ref{premul} as \emph{naturality properties} of the $M_{\mathbb{R},\Gamma_E}(C,P,D)$ matrices.

\begin{corollary}
\label{unimodular}
Let $C$ be an Euler system of a 4-regular graph $F$, and let $\Gamma_E$ be a based set of oriented induced circuits of $C$. Then $\det M_{\mathbb{R},\Gamma_E}(C,P,D) \in \{-1,0,1\}$ for every circuit partition $P$ of $F$ and every directed version $D$ of $Tch(P)$.
\end{corollary}

\begin{proof}
If $P$ is not an Euler system then $Tch(P)$ has at least one non-loop edge, so the dimension of the cycle space of $Tch(P)$ is $<|V(F)|$. According to Theorem~\ref{space}, this dimension is the same as the rank of $M_{\mathbb{R},\Gamma_E}(C,P,D)$, so $\det M_{\mathbb{R},\Gamma_E}(C,P,D)=0$.

If $P=\widetilde{C}$ is an Euler system and $D=D_{\widetilde{\Gamma}_E}$, Theorem~\ref{premul} tells us that $M_{\mathbb{R},\widetilde{\Gamma}_E}(C,P,D)^{-1}$ is a matrix of integers; this can only happen if $\det M_{\mathbb{R},\widetilde{\Gamma}_E}(C,P,D)=\pm 1$. If $D \neq D_{\widetilde{\Gamma}_E}$ then as mentioned a couple of paragraphs after Theorem~\ref{space}, $M_{\mathbb{R},\widetilde{\Gamma}_E}(C,P,D)$ is transformed into $M_{\mathbb{R},\widetilde{\Gamma}_E}(C,P,D_{\widetilde{\Gamma}_E})$ by multiplying some columns by $-1$. Consequently $\det M_{\mathbb{R},\widetilde{\Gamma}_E}(C,P,D)=\pm \det M_{\mathbb{R},\widetilde{\Gamma}_E}(C,P,D_{\widetilde{\Gamma}_E})=\pm \pm 1 = \pm 1$.
\end{proof}
Corollary~\ref{unimodular} implies that for based sets of induced circuits, Theorem~\ref{space} holds over all fields.
\begin{theorem}
\label{arbspace}
Let $C$ be an Euler system of a 4-regular graph $F$, and let $\Gamma_E$ be a based set of oriented induced circuits of $C$. Let $P$ be a circuit partition of $F$ and $D$ a directed version of $Tch(P)$. Then for every field $\mathbb{F}$, the row space of $M_{\mathbb{R},\Gamma_E}(C,P,D)$ over $\mathbb{F}$ is equal to the cycle space of $Tch(P)$ over $\mathbb{F}$.
\end{theorem}
\begin{proof}
Proposition~\ref{arbfield} tells us that the row space of $M_{\mathbb{R},\Gamma_E}(C,P,D)$ over $\mathbb{F}$ is contained in the cycle space of $Tch(P)$ over $\mathbb{F}$. To prove that the two spaces are equal, then, it suffices to prove that they have the same dimension.

As Theorem~\ref{space} holds over~$\mathbb{R}$, the dimension of the cycle space of $Tch(P)$ over $\mathbb{R}$ equals the rank of $M_{\mathbb{R},\Gamma_E}(C,P,D)$ over $\mathbb{R}$. It is well known that the dimension of the cycle space of a graph equals the number of edges excluded from a maximal forest; in particular, this dimension is the same over all fields. Consequently proving the proposition reduces to proving that  $M_{\mathbb{R},\Gamma_E}(C,P,D)$ has the same rank over $\mathbb{F}$ and $\mathbb{R}$.

The smallest possible value of $|P|$, the number of circuits in $P$, is $c(F)$. If $|P|=c(F)$ then $P$ is an Euler system of $F$, and Corollary~\ref{unimodular} tells us that $\det M_{\mathbb{R},\Gamma_E}(C,P,D)= \pm 1$ is nonzero in both $\mathbb{F}$ and $\mathbb{R}$. Consequently the rank of $M_{\mathbb{R},\Gamma_E}(C,P,D)$ is $n$ over both $\mathbb{F}$ and $\mathbb{R}$.

We proceed using induction on $|P|>c(F)$. There must be a component of $F$ that contains more than one circuit of $P$, and this component must contain a vertex $v$ incident on two distinct circuits of $P$. Let $P'$ be a circuit partition that involves the same transitions as $P$, except that $P'(v) \neq P(v)$. Then $P'$ includes the same circuits as $P$, except that the two circuits of $P$ incident at $v$ are united in one circuit of $P'$. (Two circuit partitions related in this way are pictured in Figure~\ref{fourfig3a1}.) It follows that $Tch(P')$ is the graph obtained from $Tch(P)$ by contracting the edge $e_v$ and replacing it with a loop, so the dimension of the cycle space of $Tch(P')$ is one more than the dimension of the cycle space of $Tch(P)$.

Let $D'$ be the directed version of $Tch(P')$ in which edge directions are inherited from $D$. As $|P'|=|P|-1$, the inductive hypothesis tells us that $M_{\mathbb{R},\Gamma_E}(C,P',D')$ has the same rank over $\mathbb{F}$ and $\mathbb{R}$. The only column of $M_{\mathbb{R},\Gamma_E}(C,P',D')$ that is not equal to the corresponding column of $M_{\mathbb{R},\Gamma_E}(C,P,D)$ is the $v$ column, so the ranks of $M_{\mathbb{R},\Gamma_E}(C,P',D')$ and $M_{\mathbb{R},\Gamma_E}(C,P,D)$ cannot differ by more than 1. It follows that the rank of $M_{\mathbb{R},\Gamma_E}(C,P,D)$ over $\mathbb{F}$ is at least the dimension of the cycle space of $Tch(P)$, which is the rank of $M_{\mathbb{R},\Gamma_E}(C,P,D)$ over $\mathbb{R}$.

On the other hand, any relation involving columns of $M_{\mathbb{R},\Gamma_E}(C,P,D)$ over $\mathbb{R}$ can be expressed with coefficients from $\mathbb{Z}$, so it reduces to a linear relation involving columns of $M_{\mathbb{R},\Gamma_E}(C,P,D)$ over $\mathbb{F}$. Consequently the rank of $M_{\mathbb{R},\Gamma_E}(C,P,D)$ over $\mathbb{R}$ is at least the rank of $M_{\mathbb{R},\Gamma_E}(C,P,D)$ over $\mathbb{F}$. We conclude that $M_{\mathbb{R},\Gamma_E}(C,P,D)$ has the same rank over $\mathbb{F}$ and $\mathbb{R}$.
\end{proof}
\begin{corollary}
The nullity of $M_{\mathbb{R},\Gamma_E}(C,P,D)$ over $\mathbb{F}$ is equal to $|P|-c(F)$.
\end{corollary}
\begin{proof}
For any graph $G$, the dimension of the cycle space is $|E(G)|-|V(G)|+c(G)$. Therefore Theorem \ref{arbspace} tells us that the rank of $M_{\mathbb{R},\Gamma_E}(C,P,D)$ is $|V(F)|-|P|+c(Tch(P))$. As $M_{\mathbb{R},\Gamma_E}(C,P,D)$ has $|V(F)|$ rows, we conclude that the nullity is $|P|-c(Tch(P))$. It is not hard to see that $c(Tch(P))=c(F)$; details are provided in~\cite[Section 2]{Tsign}. \end{proof}

\section{Completing the proof of Theorem~\ref{main}}
\label{sec:proof2}

Proposition~\ref{standardmg} and Theorem~\ref{premul} imply the following.
\begin{corollary}
\label{prepostmult}Let $E$ be a set of base edges in a 4-regular graph $F$, let $C$ and $\widetilde{C}$ be oriented Euler systems of $F$, and let $\Gamma^{o}_E$ and $\widetilde{\Gamma}^{o}_E$ be sets of based, consistently oriented induced circuits of $C$ and $\widetilde{C}$, respectively.

Then $IAS_{\widetilde{\Gamma}^{o}_E}(\widetilde{C})$ can be obtained from $M_{\mathbb{R},\widetilde{\Gamma}^{o}_E}(\widetilde{C},C,D_{\Gamma^{o}_E}) \cdot IAS_{\Gamma^{o}_E}(C)$ by multiplying some columns by $-1$.
\end{corollary}
\begin{proof}
By Theorem~\ref{premul},
$
M_{\mathbb{R},\widetilde{\Gamma}^{o}_E}(\widetilde{C},P,D_{\Gamma^{o}_E}) =
M_{\mathbb{R},\widetilde{\Gamma}^{o}_E}(\widetilde{C},C,D_{\Gamma^{o}_E}) \cdot M_{\mathbb{R},\Gamma^{o}_E}(C,P,D_{\Gamma^{o}_E})
$
for every circuit partition $P$ of $F$. Recall from Section~\ref{sec:circpart} that $M_{\mathbb{R},\widetilde{\Gamma}^{o}_E}(\widetilde{C},P,D_{\widetilde{\Gamma}^{o}_E})$ is obtained from $M_{\mathbb{R},\widetilde{\Gamma}^{o}_E}(\widetilde{C},P,D_{\Gamma^{o}_E})$ by multiplying some columns by $-1$. By Proposition~\ref{standardmg}, we have that $M_{\mathbb{R},\widetilde{\Gamma}^{o}_E}(\widetilde{C},P,D_{\widetilde{\Gamma}^{o}_E})$ and $M_{\mathbb{R},\Gamma^{o}_E}(C,P,D_{\Gamma^{o}_E})$ are the submatrices of $IAS_{\widetilde{\Gamma}^{o}_E}(\widetilde{C})$ and $IAS_{\Gamma^{o}_E}(C)$, respectively, corresponding to the transitions in $P$. Consequently, we obtain the desired result.
\end{proof}

Recall from the paragraph below Definition~\ref{isomat} that rows and columns of matrices in this paper are not ordered but instead indexed by finite sets. If we were required to preserve the column order of Definition \ref{bigmatrix}, Corollary \ref{prepostmult} would have to mention permuting columns in accordance with the $\phi,\chi,\psi$ representation of elements of $\mathfrak{T}(F)$ with respect to $C$ and $\widetilde{C}$.

As mentioned at the end of Section~\ref{sec:easy}, the only implication of Theorem~\ref{main} that has not yet been verified is $1 \Rightarrow 2$. This implication is part of the following.

\begin{corollary}
\label{bigone}
Let $\mathbb{F}$ be a field, let $E$ be a set that contains one edge from each connected component of a 4-regular graph $F$, and let $\Gamma^{o}_E$ be a set of based, consistently oriented induced circuits of an oriented Euler system $C$ of $F$. Let $M_{\mathbb{F}}[IAS_{\Gamma^{o}_E}(C)]$ denote the matroid represented by the matrix whose entries are the images in $\mathbb{F}$ of the entries of $IAS_{\Gamma^{o}_E}(C)$. Then $M_{\mathbb{F}}[IAS_{\Gamma^{o}_E}(C)]$ is a strict $3$-sheltering matroid for $\mathcal{Z}_3(\mathcal{I}(C))$. Moreover, $M_{\mathbb{F}}[IAS_{\Gamma^{o}_E}(C)]$ is independent of $C$.
\end{corollary}
\begin{proof}
Let $S$ be a subtransversal of $\mathcal{Z}_3(\mathcal{I}(C))$. To show that $\mathcal{Z}_3(\mathcal{I}((C))$ is strictly sheltered by $M_{\mathbb{F}}[IAS_{\Gamma^{o}_E}(C)]$, it suffices show that $S$ is an independent set of $M_{\mathbb{F}}[IAS_{\Gamma^{o}_E}(C)]$ if and only if $S$ is an independent set of $M_{\mathbb{R}}[IAS_{\Gamma^{o}_E}(C)]$. The only-if direction follows since $\mathbb{R}$ is of characteristic $0$. Conversely, assume that $S$ is an independent set of $M_{\mathbb{R}}[IAS_{\Gamma^{o}_E}(C)]$. Since all bases of $\mathcal{Z}_3(\mathcal{I}(C))$ are transversals, $S \subseteq B$ for some transversal $B$ that is a basis of $M_{\mathbb{R}}[IAS_{\Gamma^{o}_E}(C)]$. By Corollary~\ref{unimodular}, $B$ is also a basis of $M_{\mathbb{F}}[IAS_{\Gamma^{o}_E}(C)]$. Hence, $I$ is an independent set of $M_{\mathbb{F}}[IAS_{\Gamma^{o}_E}(C)]$.

To verify the last sentence of the statement, let $\widetilde{C}$ be another Euler system of $F$. By Corollary~\ref{prepostmult} and \cite[Section 6.3]{O}, the matroids of $IAS_{\Gamma^{o}_E}(C)$ and $IAS_{\widetilde{\Gamma}^{o}_E}(\widetilde{C})$ are equal.
\end{proof}

We should point out that although the matroid $M_{\mathbb{F}}[IAS_{\Gamma^{o}_E}(C)]$ is independent of $C$, it is not independent of $E$ or $\mathbb{F}$. An example of dependence on $E$ is given in the next section. For dependence on $\mathbb{F}$, note that the fact that $IAS_{\Gamma^{o}_E}(C)$ reduces to $IAS(\mathcal{I}(C))$ modulo 2 implies that $M[IAS(G)]=M_{GF(2)}[IAS_{\Gamma^{o}_E}(C)]$. As shown in~\cite{Tnewnew}, it follows that if $G$ has a connected component with three or more vertices then the matroid $M[IAS(G)]=M_{GF(2)}[IAS_{\Gamma^{o}_E}(C)]$ is not regular, i.e., it cannot be represented over any field of characteristic $\neq 2$. We deduce that $M_{GF(2)}[IAS_{\Gamma^{o}_E}(C)]$ cannot be isomorphic to $M_{\mathbb{F}}[IAS_{\Gamma^{o}_E}(C)]$ if $\textrm{char}(\mathbb{F}) \neq 2$.

\section{An example}
\label{sec:example}

We illustrate the above results with an example.

Figure~\ref{fourfig2a} illustrates two oriented Euler circuits in a 4-regular graph. As in Figure~\ref{fourfig3a1}, we trace an Euler circuit by walking along the edges of the graph, and maintaining the dashed/plain line status when passing through a vertex. The two illustrated Euler circuits could be represented by the double occurrence words $abcdbacd$ and $abcdcabd$, respectively.

With $E=\{ad\}$, signed versions of these Euler circuits are $C=a^{-}b^{-}c^{-}d^{-}b^{+}a^{+}c^{+}d^{+}$ and $\widetilde{C}=a^{-}b^{-}c^{-}d^{-}c^{+}a^{+}b^{+}d^{+}$. The resulting matrices are as follows.
\[
IAS_{\Gamma^{o}_{\{ad\}}}(C) =
\scalebox{0.9}{
{\let\quad\thinspace
\bordermatrix{
~ &
\phi_C(a) & \phi_C(b) & \phi_C(c) & \phi_C(d) &
\chi_C(a) & \chi_C(b) & \chi_C(c) & \chi_C(d) &
\psi_C(a) & \psi_C(b) & \psi_C(c) & \psi_C(d)\cr
a & 1 & 0 & 0 & 0 & 0 & 0 & -1 & -1 & 1 & 2 & 1 & 1\cr
b & 0 & 1 & 0 & 0 & 0 & 0 & -1 & -1 & 0 & 1 & 1 & 1\cr
c & 0 & 0 & 1 & 0 & 1 & 1 & 0 & -1 & 1 & 1 & 1 & 1\cr
d & 0 & 0 & 0 & 1 & 1 & 1 & 1 & 0 & 1 & 1 & 1 & 1
}}}
\]
\[
IAS_{\widetilde{\Gamma}^{o}_{\{ad\}}}(\widetilde{C})=
\scalebox{0.9}{
{\let\quad\thinspace
\bordermatrix{
~ &
\phi_{\widetilde{C}}(a) & \phi_{\widetilde{C}}(b) & \phi_{\widetilde{C}}(c) & \phi_{\widetilde{C}}(d) &
\chi_{\widetilde{C}}(a) & \chi_{\widetilde{C}}(b) & \chi_{\widetilde{C}}(c) & \chi_{\widetilde{C}}(d) &
\psi_{\widetilde{C}}(a) & \psi_{\widetilde{C}}(b) & \psi_{\widetilde{C}}(c) & \psi_{\widetilde{C}}(d)\cr
a & 1 & 0 & 0 & 0 & 0 & -1 & 0 & -1 & 1 & 1 & 2 & 1\cr
b & 0 & 1 & 0 & 0 & 1 & 0 & 0 & -1 & 1 & 1 & 2 & 1\cr
c & 0 & 0 & 1 & 0 & 0 & 0 & 0 & -1 & 0 & 0 & 1 & 1\cr
d & 0 & 0 & 0 & 1 & 1 & 1 & 1 & 0 & 1 & 1 & 1 & 1
}}}
\]

By Corollary~\ref{prepostmult} we can obtain $IAS_{\Gamma^{o}_{\{ad\}}}(C)$ from $M_{\mathbb{R},\Gamma^{o}_{\{ad\}}}(C,\widetilde{C},D_{\widetilde{\Gamma}^{o}_{\{ad\}}}) \cdot IAS_{\widetilde{\Gamma}^{o}_{\{ad\}}}(\widetilde{C})$ by multiplying some columns by $-1$ to reflect the differences in edge directions of $D_{\Gamma^{o}_{\{ad\}}}$ and $D_{\widetilde{\Gamma}^{o}_{\{ad\}}}$. We proceed to verify this assertion.

Note that $\widetilde{C}=C*d$. Since $\phi_{\widetilde{C}}(d) = \psi_{C}(d)$ and $\chi_{\widetilde{C}}(x) = \psi_{C}(x)$ for all $x \in \{a,b,c\}$ (and the same identities when interchanging $C$ and $\widetilde{C}$), we see that
\[
IAS_{\widetilde{\Gamma}^{o}_{\{ad\}}}(\widetilde{C}) =
\scalebox{0.9}{
{\let\quad\thinspace
\bordermatrix{
~ &
\phi_C(a) & \phi_C(b) & \phi_C(c) & \phi_C(d) &
\chi_C(a) & \chi_C(b) & \chi_C(c) & \chi_C(d) &
\psi_C(a) & \psi_C(b) & \psi_C(c) & \psi_C(d)\cr
a & 1 & 0 & 0 & 1 & 1 & 1 & 2 & -1 & 0 & -1 & 0 & 0\cr
b & 0 & 1 & 0 & 1 & 1 & 1 & 2 & -1 & 1 & 0 & 0 & 0\cr
c & 0 & 0 & 1 & 1 & 0 & 0 & 1 & -1 & 0 & 0 & 0 & 0\cr
d & 0 & 0 & 0 & 1 & 1 & 1 & 1 & 0 & 1 & 1 & 1 & 1
}}}.
\]
Multiplying $IAS_{\widetilde{\Gamma}^{o}_{\{ad\}}}(\widetilde{C})$ on the left by
\[
M_{\mathbb{R},\Gamma^{o}_{\{ad\}}}(C,C*d,D_{\widetilde{\Gamma}^{o}_{\{ad\}}}) =
\bordermatrix{
~ & a & b & c & d \cr
a & 1 & 0 & 0 & -1 \cr
b & 0 & 1 & 0 & -1 \cr
c & 0 & 0 & 1 & -1 \cr
d & 0 & 0 & 0 & -1
}
\]
yields
\[
\scalebox{0.9}{
{\let\quad\thinspace
\bordermatrix{
~ &
\phi_C(a) & \phi_C(b) & \phi_C(c) & \phi_C(d) &
\chi_C(a) & \chi_C(b) & \chi_C(c) & \chi_C(d) &
\psi_C(a) & \psi_C(b) & \psi_C(c) & \psi_C(d)\cr
a & 1 & 0 & 0 & 0 & 0 & 0 & 1 & -1 & -1 & -2 & -1 & -1\cr
b & 0 & 1 & 0 & 0 & 0 & 0 & 1 & -1 & 0 & -1 & -1 & -1\cr
c & 0 & 0 & 1 & 0 & -1 & -1 & 0 & -1 & -1 & -1 & -1 & -1\cr
d & 0 & 0 & 0 & -1 & -1 & -1 & -1 & 0 & -1 & -1 & -1 & -1
}}}.
\]
To obtain $IAS_{\Gamma^{o}_{\{ad\}}}(C)$ multiply every column containing a nonzero entry in the last row by $-1$.

For later comparison we give now the circuits of cardinality 3 of $M_{\mathbb{R}}[IAS_{\Gamma^{o}_{\{ad\}}}(C)]$, in a compressed way for readability:
\begin{align*}
\phi_C(a)\chi_C(a)\psi_C(a),& & \phi_C(a)\chi_C(b)\psi_C(a),& & \phi_C(d)\chi_C(d)\psi_C(c),& & \phi_C(d)\chi_C(d)\psi_C(d),\\
\phi_C(a)\psi_C(b)\psi_C(c),& & \phi_C(a)\psi_C(b)\psi_C(d),& & \phi_C(b)\psi_C(a)\psi_C(c),& & \phi_C(b)\psi_C(a)\psi_C(d), \\
\phi_C(c)\phi_C(d)\chi_C(a),& & \phi_C(c)\phi_C(d)\chi_C(b),& & \chi_C(a)\chi_C(c)\chi_C(d),& & \chi_C(b)\chi_C(c)\chi_C(d).
\end{align*}

The last eight of these listed 3-circuits are transverse circuits; as discussed in~\cite{BT2}, they correspond to the eight 3-cycles in $F$. (Each set of three vertices in $F$ appears on two 3-cycles.) Notice that each of $\chi_C(a), \chi_C(b), \psi_C(c), \psi_C(d)$ appears in precisely three 3-circuits.

\begin{figure}[ptb]
\centering
\includegraphics[scale=0.8]{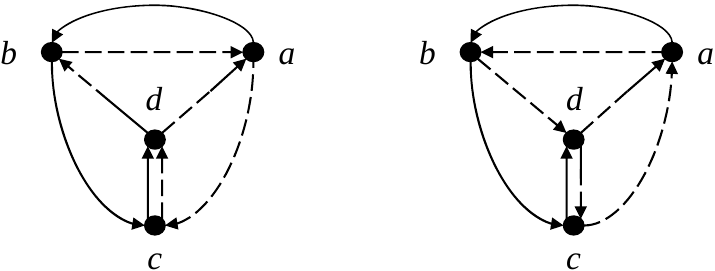}
\caption{Two oriented Euler circuits in a 4-regular graph.}
\label{fourfig2a}
\end{figure}

If we use one of the $cd$ edges as the base edge instead of $ad$, then we obtain signed Euler circuits $C = d^{-}a^{-}b^{-}c^{-}d^{+}b^{+}a^{+}c^{+}$
and $\widetilde{C} = d^{-}c^{-}b^{-}a^{-}d^{+}b^{+}a^{+}c^{+}$. The resulting matrices are given below.
\[
IAS_{\Gamma^{o}_{\{cd\}}}(C)=
\scalebox{0.9}{
{\let\quad\thinspace
\bordermatrix{
~ &
\phi_C(a) & \phi_C(b) & \phi_C(c) & \phi_C(d) &
\chi_C(a) & \chi_C(b) & \chi_C(c) & \chi_C(d) &
\psi_C(a) & \psi_C(b) & \psi_C(c) & \psi_C(d)\cr
a & 1 & 0 & 0 & 0 & 0 & 0 & -1 & 1 & 1 & 2 & 1 & 1\cr
b & 0 & 1 & 0 & 0 & 0 & 0 & -1 & 1 & 0 & 1 & 1 & 1\cr
c & 0 & 0 & 1 & 0 & 1 & 1 & 0 & 1 & 1 & 1 & 1 & 1\cr
d & 0 & 0 & 0 & 1 & -1 & -1 & -1 & 0 & 1 & 1 & 1 & 1
}}}
\]
\[
IAS_{\widetilde{\Gamma}^{o}_{\{cd\}}}(\widetilde{C})=
\scalebox{0.9}{
{\let\quad\thinspace
\bordermatrix{
~ &
\phi_{\widetilde{C}}(a) & \phi_{\widetilde{C}}(b) & \phi_{\widetilde{C}}(c) & \phi_{\widetilde{C}}(d) &
\chi_{\widetilde{C}}(a) & \chi_{\widetilde{C}}(b) & \chi_{\widetilde{C}}(c) & \chi_{\widetilde{C}}(d) &
\psi_{\widetilde{C}}(a) & \psi_{\widetilde{C}}(b) & \psi_{\widetilde{C}}(c) & \psi_{\widetilde{C}}(d)\cr
a & 1 & 0 & 0 & 0 & 0 & 1 & 0 & 1 & 1 & 1 & 0 & 1\cr
b & 0 & 1 & 0 & 0 & -1 & 0 & 0 & 1 & 1 & 1 & 0 & 1\cr
c & 0 & 0 & 1 & 0 & 0 & 0 & 0 & 1 & 2 & 2 & 1 & 1\cr
d & 0 & 0 & 0 & 1 & -1 & -1 & -1 & 0 & 1 & 1 & 1 & 1
}}}
\]

We have $M_{\mathbb{R},\Gamma^{o}_{\{cd\}}}(C,C*d,D_{\widetilde{\Gamma}^{o}_{\{cd\}}}) = M_{\mathbb{R},\Gamma^{o}_{\{ad\}}}(C,C*d,D_{\widetilde{\Gamma}^{o}_{\{ad\}}})$ and $IAS_{\Gamma^{o}_{\{cd\}}}(C)$ is obtained from $M_{\mathbb{R},\Gamma^{o}_{\{cd\}}}(C,C*d,D_{\widetilde{\Gamma}^{o}_{\{cd\}}}) \cdot IAS_{\widetilde{\Gamma}^{o}_{\{cd\}}}(\widetilde{C})$ by multiplying by $-1$ the columns indexed by $\psi_{\widetilde{C}}(d) = \phi_C(d)$ and $\phi_{\widetilde{C}}(d) = \psi_{C}(d)$ (recall that in the case of base edge $\{ad\}$ we needed to multiply eight columns by $-1$).

We also verify that the matroids $M_{\mathbb{R}}[IAS_{\Gamma^{o}_{\{cd\}}}(C)]$ and $M_{\mathbb{R}}[IAS_{\Gamma^{o}_{\{ad\}}}(C)]$ are not isomorphic. Indeed, the 3-circuits of $M_{\mathbb{R}}[IAS_{\Gamma^{o}_{\{cd\}}}(C)]$ are:
\begin{align*}
\phi_C(a)\psi_C(b)\psi_C(c),& &\phi_C(a)\psi_C(b)\psi_C(d),& & \phi_C(b)\psi_C(a)\psi_C(c),& &\phi_C(b)\psi_C(a)\psi_C(d), \\ \phi_C(c)\phi_C(d)\chi_C(a),& &\phi_C(c)\phi_C(d)\chi_C(b),& & \phi_C(c)\chi_C(c)\psi_C(c),& &\phi_C(c)\chi_C(c)\psi_C(d), \\ \phi_C(d)\chi_C(d)\psi_C(c),& &\phi_C(d)\chi_C(d)\psi_C(d),& & \chi_C(a)\chi_C(c)\chi_C(d),& &\chi_C(b)\chi_C(c)\chi_C(d).
\end{align*}

Unlike $M_{\mathbb{R}}[IAS_{\Gamma^{o}_{\{ad\}}}(C)]$, this matroid has no element that appears in precisely three 3-circuits (in fact, each element appears in an even number of 3-circuits).

\section{Naji's Theorem}
\label{sec:naji}
In this section we discuss a characterization of circle graphs discovered by Naji~\cite{N1, N}.

\begin{definition}
\label{najieq} Let $G$ be a simple graph. For each pair of
distinct vertices $v$ and $w$ of $G$, let $\beta(v,w)$ and $\beta(w,v)$ be
distinct variables. Then the \emph{Naji equations} for $G$ are the following.
\begin{enumerate}
\item If $vw$ is an edge of $G$ then $\beta(v,w)+\beta(w,v)=1$.

\item If $v,w,x$ are three distinct vertices of $G$ such that $vw$ is an edge of $G$ and $vx,wx$ are not edges of $G$, then $\beta(x,v)+\beta(x,w)=0$.

\item If $v,w,x$ are three distinct vertices of $G$ such that $vw,vx$ are edges of $G$ and $wx$ is not an edge of $G$, then $\beta(v,w)+\beta(v,x)+\beta(w,x)+\beta(x,w)=1$.
\end{enumerate}
\end{definition}

\begin{theorem}(\cite{GeelenNaji, N1, N, Tn}) $G$ is a circle graph if and only if the Naji equations of $G$ have a solution over $GF(2)$.
\end{theorem}

It turns out that the $IAS_{\Gamma^{o}}(C)$ matrices of Section~\ref{sec:easy} provide solutions to the Naji equations for the interlacement graph $\mathcal{I}(C)$.

\begin{proposition}
\label{najisol}
Let $C$ be an Euler system of a 4-regular graph $F$, and $\Gamma^{o}$ a set of oriented induced circuits of $C$. Let $IAS_{\Gamma^{o}}(C)=
\begin{pmatrix}
I & A & B
\end{pmatrix}$ be a corresponding matrix defined as in Section~\ref{sec:easy}. For $v \neq w \in V(F)$ define $\beta(v,w) \in GF(2)$ as follows.
\begin{gather*}
\beta(v,w)=
\begin{cases}
0\text{,} & \text{if }vw \in E(\mathcal{I}(C)) \text{ and }A_{vw}=1\\
1\text{,} & \text{if }vw \in E(\mathcal{I}(C)) \text{ and }A_{vw}=-1\\
0\text{,} & \text{if }vw \notin E(\mathcal{I}(C)) \text{ and }B_{vw}=0\\
1\text{,} & \text{if }vw \notin E(\mathcal{I}(C)) \text{ and }B_{vw}=2
\end{cases}
\end{gather*}
Then $\beta$ is a Naji solution for $\mathcal{I}(C)$.
\end{proposition}
\begin{proof}
The verification is routine. For details see~\cite[Proposition 4]{Tn}, where the notation $v^{+}=v^{in},v^{-}=v^{out}$ is used.
\end{proof}

The solution described in Proposition~\ref{najisol} cannot be extracted from the $GF(2)$ reduction of $IAS_{\Gamma^{o}}(C)$, because $-1 \equiv 1$ and $0 \equiv 2$ (mod 2). We come to the rather curious conclusion that even though the Naji equations are defined over $GF(2)$, they are connected to representations of $\mathcal{Z}_{3}(G)$ over fields of characteristic other than 2.

\section{Characterizations of circle graphs and planar matroids in terms of multimatroids}
\label{sec:other}
In this section we formulate a more detailed form of Theorem~\ref{main} in terms of multimatroids.

\subsection{Multimatroids}

First we recall some basic notions regarding multimatroids from \cite{B1,RB1}.

The power set of a set $X$ is denoted by $2^X$. Let $\Omega$ be a partition of a set $U$. A \emph{transversal} $T$ of $\Omega$ is a subset of $U$ such that $|T \cap \omega| = 1$ for all $\omega \in \Omega$. The set of transversals of $\Omega$ is denoted by $\mathcal{T}(\Omega)$. A \emph{subtransversal} of $\Omega$ is a subset of a transversal of $\Omega$. The set of subtransversals of $\Omega$ is denoted by $\mathcal{S}(\Omega)$.

In order to efficiently define the notion of a multimatroid, we first recall the notion of a semi-multimatroid.
\begin{definition} \label{def:semi-multimatroid}
A \emph{semi-multimatroid} $Z$ (described by its circuits) is a triple $(U,\Omega,\mathcal{C})$, where $\Omega$ is a partition of a finite set $U$ and $\mathcal{C} \subseteq \mathcal{S}(\Omega)$ such that for each $T \in \mathcal{T}(\Omega)$, $(T,\mathcal{C}\cap 2^T)$ is a matroid (described by its circuits).
\end{definition}
For a semi-multimatroid $Z = (U,\Omega,\mathcal{C})$, we denote $U(Z) := U$, $\Omega(Z) := \Omega$, and $\mathcal{C}(Z) := \mathcal{C}$. If $Z$ is clear from the context, then we just write $U$, $\Omega$ and $\mathcal{C}$ to denote $U(Z)$, $\Omega(Z)$, and $\mathcal{C}(Z)$, respectively. We say that $Z$ is a semi-multimatroid \emph{on} $(U,\Omega)$.

The elements of $\Omega$ are called the \emph{skew classes} of $Z$, the elements of $\mathcal{C}$ are called the \emph{circuits} of $Z$, and $U$ is called the \emph{ground set} of $Z$. We say that $I \in \mathcal{S}(\Omega)$ is an \emph{independent set} of $Z$ is no subset of $I$ is a circuit and we say that $B \in \mathcal{S}(\Omega)$ is a \emph{basis} of $Z$ if $B$ is an independent set, but no proper superset of $B$ is an independent set. The \emph{order} of $Z$ is $|\Omega(Z)|$.

For any $X \subseteq U$, the \emph{restriction} of $Z$ to $X$, denoted by $Z[X]$, is the semi-multimatroid $(X,\Omega',\mathcal{C}\cap 2^{X})$ with $\Omega' = \{ \omega \cap X \mid \omega \in \Omega\} \setminus \{\emptyset\}$. We also define $Z-X := Z[U\setminus X]$.

If all elements of $\Omega(Z)$ are singletons, then, by slight abuse of notation, we associate $Z$ with the matroid $(U(Z),\mathcal{C}(Z))$.

For $X \in \mathcal{S}(\Omega)$, the \emph{minor} of $Z$ induced by $X$ is the semi-multimatroid $Z|X := (U',\Omega',\mathcal{C}')$, where $\Omega' = \{\omega \in \Omega \mid \omega \cap X = \emptyset\}$, $U' = \bigcup \Omega'$, and for all $T' \in \mathcal{T}(\Omega')$, $(Z|X)[T'] = Z[T' \cup X] \slash X$, where $Z[T' \cup X]$ is regarded here as a matroid and $\slash X$ denotes, as usual, matroid contraction of $X$. Explicitly, $\mathcal{C}' = \{ C \in \mathcal{C}(Z[T' \cup X] \slash X) \mid T' \in \mathcal{T}(\Omega') \}$.

We remark that, unfortunately, the standard way to denote a multimatroid
minor (i.e., $Z|X$) as introduced in \cite{B2} clashes with the usual way to denote matroid restriction. To avoid confusion, we denote in this paper the restriction
of a matroid $M$ to a subset $X$ of its ground set by $M[X]$ (which is compatible with the notation of multimatroid restriction $Z[X]$).

Semi-multimatroids $Z$ and $Z'$ are called \emph{isomorphic} if there is a one-to-one correspondence between $U(Z)$ and $U(Z')$ respecting $\Omega$ and $\mathcal{C}$.

For a semi-multimatroid $Z$ and matroid $M$ with ground set $U(Z)$, we say that $Z$ is \emph{sheltered} by $M$ if $Z[T]=M[T]$ for all $T \in \mathcal{T}(\Omega(Z))$. We say that a matrix $A$ \emph{represents} (or, is a \emph{representation} of) a semi-multimatroid $Z$ if $A$ represents a matroid $M$ that shelters $Z$ \cite{BT1}. We say that a semi-multimatroid $Z$ is \emph{representable} over some field $\mathbb{F}$ if there is a matrix $A$ over $\mathbb{F}$ that represents $Z$. If $Z$ is representable over $\mathbb{F}$, then so is every minor of $Z$. A semi-multimatroid is called \emph{regular} if it is representable over all fields. We say that $A$ is a \emph{strict} representation of $Z$ if $A$ has at most $|\Omega(Z)|$ rows. This terminology is consistent with \cite{BT1}, where it is said that $M$ is a strict sheltering matroid if $r(M) \leq |\Omega(Z)|$. We note that this terminology is also consistent with the notion of strict for sheltering matroids of $\mathcal{Z}_{2}(G)$ or $\mathcal{Z}_{3}(G)$ as defined in Section~\ref{sec:reps}, since both $\mathcal{Z}_{2}(G)$ and $\mathcal{Z}_{3}(G)$ have bases that are transversals and so $r(M) < |\Omega(Z)|$ is impossible for sheltering matroids $M$ of $Z$ when $Z$ is equal to $\mathcal{Z}_{2}(G)$ or $\mathcal{Z}_{3}(G)$.

\begin{definition} \label{def:multimatroid}
A semi-multimatroid $Z$ is called a \emph{multimatroid} if every minor of $Z$ of order $1$ has at most one circuit. A multimatroid is called \emph{tight} if every minor of $Z$ of order $1$ has exactly one circuit.
\end{definition}
For instance, suppose $Z$ is a semi-multimatroid of order 1, i.e., $Z$ contains a single skew class. Thus $Z$ is of the form $(\omega,\{\omega\},\mathcal{C})$ where $\mathcal{C}$ is a set of singletons. Then by Definition~\ref{def:multimatroid}, $Z$ is a multimatroid if and only if $|\mathcal{C}| \leq 1$, and $Z$ is a tight multimatroid if and only if $|\mathcal{C}| = 1$.

We remark that we use here the slightly more liberal notion of tightness from \cite{RB1}, while the notion of tightness from \cite{B1} additionally requires that no element of $\Omega(Z)$ is a singleton.

By definition, both multimatroids and tight multimatroids are closed under taking minors. Also, if $Z$ is a multimatroid and $X \subseteq U(Z)$, then so is $Z[X]$. However, tight multimatroids are in general not closed under taking restrictions.

If $Z$ is a multimatroid where no element of $\Omega(Z)$ is a singleton, then all bases of $Z$ are transversals \cite[Proposition~5.5]{B1}. If all elements of $\Omega(Z)$ for some multimatroid $Z$ are of cardinality $q$, then $Z$ is called a \emph{$q$-matroid}.

It is shown in \cite{BT1} that for every looped simple graph $G$, $\mathcal{Z}_3(G)$ is a tight 3-matroid representable over $GF(2)$ and, conversely, every tight 3-matroid representable over $GF(2)$ is isomorphic to $\mathcal{Z}_3(G)$ for some looped simple graph.

\subsection{Main result}

We recall the following known result.
\begin{lemma} [\cite{Tnewnew}] \label{lem:vertex_minor_mm_minor}
Let $G$ be a looped simple graph. If $H$ is a vertex-minor of $G$, then $\mathcal{Z}_{3}(H)$ is isomorphic to a minor of $\mathcal{Z}_{3}(G)$. Conversely, if $Z$ is a minor of $\mathcal{Z}_{3}(G)$, then $Z$ is isomorphic to $\mathcal{Z}_{3}(H)$ for some vertex-minor $H$ of $G$.
\end{lemma}
\begin{proof}
Let $H$ be a vertex-minor of $G$. Then there is a graph $G'$ locally equivalent to $G$ with $G'-X = H$ for some $X \subseteq V(G')$. By \cite[Theorem~2]{Tnewnew}, $\mathcal{Z}_{3}(G')$ is isomorphic to $\mathcal{Z}_{3}(G)$. By \cite[Proposition~35]{Tnewnew}, $\mathcal{Z}_{3}(H)$ is equal to a minor of $\mathcal{Z}_{3}(G')$.

Let $Z$ be a minor of $\mathcal{Z}_{3}(G)$. Then by \cite[Proposition~35 and Corollary~38]{Tnewnew}, $Z$ is isomorphic to $\mathcal{Z}_{3}(H)$ for some vertex-minor $H$ of $G$.
\end{proof}

Let $\Omega$ be a set of three disjoint circuits of the matroid $AG(2,3)$, the rank-$3$ affine geometry over $GF(3)$. Then $AG(2,3)$ shelters a tight $3$-matroid $\mathcal{H}_{3,3}$ with $\Omega(\mathcal{H}_{3,3}) = \Omega$, see \cite{RB1}. Explicitly, we have $\Omega(\mathcal{H}_{3,3}) = \{\omega_a,\omega_b,\omega_c\}$, where $\omega_x = \{x_1,x_2,x_3\}$ for $x \in \{a,b,c\}$ and
\begin{align*}
\mathcal{C}(\mathcal{H}_{3,3}) = \{&
\{a_1,b_1,c_1\}, \{a_2,b_2,c_2\}, \{a_3,b_3,c_3\}, \cr
& \{a_1,b_3,c_2\}, \{a_2,b_1,c_3\}, \{a_3,b_2,c_1\}, \cr
& \{a_1,b_2,c_3\}, \{a_2,b_3,c_1\}, \{a_3,b_1,c_2\}
\}.
\end{align*}

We say that a strict representation $A$ of $Z$ over $\mathbb{R}$ is \emph{transversely unimodular} if for each transversal $T$ of $Z$, the square matrix obtained from $A$ has determinant $0$, $-1$, or $1$.

We are now ready to prove a detailed form of Theorem~\ref{main} in terms of multimatroids.
\begin{theorem} \label{thm:circle_graph_mm_reg}
Let $Z$ be a tight $3$-matroid. Then the following conditions are equivalent.
\begin{enumerate}
\item \label{item:transv_unim} $Z$ has a transversely unimodular strict representation with only integer entries.

\item \label{item:regular} $Z$ is regular (i.e., representable over all fields).

\item \label{item:repr_GF2_diff} $Z$ is representable over $GF(2)$ and over some field of characteristic different from $2$.

\item \label{item:min_T_reg} For all transversals $T$ of $Z$, the 2-matroid $Z-T$ is representable over $GF(2)$ and over some field of characteristic different from $2$.

\item \label{item:excl_minor} $Z$ does not have a minor isomorphic to $\mathcal{H}_{3,3}$, $\mathcal{Z}_{3}(W_5)$, $\mathcal{Z}_{3}(W_7)$, or $\mathcal{Z}_{3}(BW_3)$.

\item \label{item:isom_circle_graph} $Z$ is isomorphic to $\mathcal{Z}_{3}(G)$ for some circle graph $G$.

\end{enumerate}
\end{theorem}
\begin{proof}
Assume Condition~\ref{item:transv_unim} holds. Let $A$ be a transversely unimodular strict representation of $Z$ over $\mathbb{R}$ with only integer entries. Let $S$ be a subtransversal of $Z$. We show that the columns of $S$ in $A$ are independent when computed over $\mathbb{R}$ if and only if these columns are independent when computed over some field $\mathbb{F}$. The if direction follows directly from the fact that $A$ has only integer entries and the fact that $\mathbb{R}$ is of characteristic zero. Conversely, assume that the columns of $S$ in $A$ are independent when computed over $\mathbb{R}$. Then $S$ is an independent set of $Z$. Since all bases of a $3$-matroid are transversals, there is a basis $B$ of $Z$ that is a transversal and $S \subseteq B$. Since the square matrix obtained from $A$ by restricting to the columns of $B$ is unimodular, these columns of $B$ are also independent when computed over $\mathbb{F}$. Thus the columns of $I \subseteq B$ are independent when computing over $\mathbb{F}$.

Condition~\ref{item:regular} directly implies Condition~\ref{item:repr_GF2_diff}, which in turn directly implies Condition~\ref{item:min_T_reg}.

Assume Condition~\ref{item:min_T_reg} holds. Assume $Z$ has a minor $Z|X$ isomorphic to $\mathcal{Z}_{3}(G)$ for some $G \in \{W_5,BW_3,W_7\}$. By Proposition~\ref{obstruct}, there is a transversal $T$ of $Z|X$ such that $(Z|X)-T$ is not representable over any field of characteristic different from $2$. Thus $Z|X$, and therefore also $Z$, are not representable over any field of characteristic different from $2$ --- a contradiction.

Now assume $Z$ has a minor $Z|X$ isomorphic to $\mathcal{H}_{3,3}$. In \cite{RB1} it is shown that there is a transversal $T$ of $\mathcal{H}_{3,3}$, such that $\mathcal{H}_{3,3}-T$ is isomorphic to the 2-matroid $S_1 = (U,\Omega,\mathcal{C})$ with $\Omega = \{ \{i_a, i_b\} \mid i \in \{1, 2, 3\} \}$, $U = \bigcup \Omega$, and $\mathcal{C} = \{ \{1_a,2_b,3_b\}, \{1_b,2_a,3_b\}, \{1_b,2_b,3_a\} \}$. Now, $S_1$ is not representable over $GF(2)$. Indeed, suppose to the contrary that $S_1$ is representable over $GF(2)$. Then the symmetric difference of the three circuits of $\mathcal{C}$, which is $\{1_a,2_a,3_a\}$, is the union of disjoint circuits of a binary matroid that shelters $S_1$. Since $\{1_a,2_a,3_a\}$ is a subtransversal, these circuits must be in $\mathcal{C}$. This is a contradiction. Consequently, $Z|X$, and therefore also $Z$, are not representable over $GF(2)$ --- a contradiction. We have thus obtained Condition~\ref{item:excl_minor}.

Assume that Condition~\ref{item:excl_minor} holds. In \cite{RB1} it is shown that if $Z$ does not have a minor isomorphic to $\mathcal{H}_{3,3}$, then $Z$ is representable over $GF(2)$. In \cite{BT1} it is shown that a tight 3-matroid representable over $GF(2)$ is isomorphic to $\mathcal{Z}_3(G)$ for some looped simple graph. By the definition of an isotropic matroid, we may assume that $G$ has no loops. By Theorem~\ref{Bthm} and Lemma~\ref{lem:vertex_minor_mm_minor}, $G$ is a circle graph. Thus Condition~\ref{item:isom_circle_graph} holds.

Condition~\ref{item:isom_circle_graph} implies Condition~\ref{item:transv_unim} by Corollaries~\ref{unimodular} and \ref{bigone}.
\end{proof}   

\begin{remark}
The equivalence of Conditions~\ref{item:min_T_reg} and \ref{item:isom_circle_graph} in Theorem~\ref{thm:circle_graph_mm_reg} is closely related to \cite[Theorem~3]{G} and \cite[Corollary~4.20]{GeelenPhD}, two positive solutions of Bouchet's conjecture that a graph is a circle graph if and only if every locally equivalent graph has a unimodular orientation \cite{Bu}. While Gasse's statement involves unimodular orientations and Geelen's statement involves delta-matroid representations, the result can also be formulated, see \cite[Theorem~25]{BT2}, in terms of multimatroids as follows: for some graph $G$, we have that $G$ is a circle graph if and only if for all transversals $T$ of $\mathcal{Z}_{3}(G)$, if the 2-matroid $\mathcal{Z}_{3}(G)-T$ is tight, then $\mathcal{Z}_{3}(G)-T$ has a transversely unimodular representation $(I \ A)$ over $\mathbb{R}$, where $A$ is skew-symmetric.
Since $\mathcal{Z}_{2}(W_5)$, $\mathcal{Z}_{2}(W_7)$, and $\mathcal{Z}_{2}(BW_3)$ are tight, the restriction to 2-matroids $\mathcal{Z}_{3}(G)-T$ that are tight is justified. The equivalence of Conditions~\ref{item:min_T_reg} and \ref{item:isom_circle_graph} in Theorem~\ref{thm:circle_graph_mm_reg} shows that the skew-symmetry restriction of $A$ can be replaced by the more liberal condition that $A$ has only integer entries (note that the skew-symmetry condition of $A$ along with transversal unimodularity of $(I \ A)$ implies that $A$ has only integer entries).

We might also mention a remark of Bouchet \cite[p.\ 285]{B4}, that if $F$ is a 4-regular graph and $C$ an Euler system, then the 2-matroid obtained from $\mathcal Z_3(\mathcal I (C)) - \{ \psi_{\mathcal I (C)}(v) \mid v \in V(F) \}$ is representable over arbitrary fields. (Bouchet used a different notion of multimatroid representability than we do, but the difference is not significant in the present discussion.) Bouchet did not extend this remark to $Z_3(\mathcal I (C)) - T$ for other transversals $T$ (or to $Z_3(\mathcal I (C))$ itself), or observe that the extended property characterizes circle graphs.
\end{remark}

\subsection{A characterization of planarity}

This subsection shows that the notion of regularity of tight 3-matroids generalizes the notion of planarity of matroids.

For a matroid $M$, denote by $E(M)$ the set of elements of $M$. For an injective function $\varphi$ on $E(M)$, denote by $\varphi(M)$ the matroid obtained from $M$ by renaming $E(M)$ according to $\varphi$.

\begin{definition} \label{def:free_sum}
Let $M$ be a matroid and let $\varphi_i$, with $i \in \{1,2\}$, be the function which sends every $e \in E(M)$ to $(e,i)$. Let $\Omega = \{ \{(e,1),(e,2)\} \mid e \in E(M) \}$ and $U = \bigcup \Omega$. Then the semi-multimatroid on $(U,\Omega)$ that is sheltered by the direct sum of the matroids $\varphi_1(M^*)$ and $\varphi_2(M)$ is denoted by $\mathcal{Z}_2(M)$.
\end{definition}
It turns out that for every matroid $M$, $\mathcal{Z}_2(M)$ is, in fact, a tight multimatroid \cite{B2}. We remark that $\mathcal{Z}_{2}(BW_3) = \mathcal{Z}_{2}(F)$, where $F$ is the Fano matroid, see, e.g., \cite{GeelenPhD}.

For every binary matroid $M$, there is a unique tight 3-matroid $\mathcal{Z}_{3}(M)$ on $(U,\Omega)$ with $\Omega = \{ \{(e,1),(e,2),(e,3)\} \mid e \in E(M) \}$ such that $\mathcal{Z}_{3}(M) - T = \mathcal{Z}_2(M)$, where $T = \{(e,3) \mid e \in E(M)\}$ \cite{BH3}. In fact, this holds even when $M$ is quaternary \cite{RB1}. The uniqueness of $\mathcal{Z}_{3}(M)$ follows from the following result.

\begin{proposition} [\cite{BH3}] \label{prop:unique_tight_k+1_mm}
Let $Z_1 = (U,\Omega,\mathcal{C}_1)$ and $Z_2 = (U,\Omega,\mathcal{C}_2)$ be tight multimatroids with $|\omega| \geq 3$ for all $\omega \in \Omega$. Let $T \in \mathcal{T}(\Omega)$. If $Z_1 - T = Z_2 - T$, then $Z_1 = Z_2$.
\end{proposition}

The next lemma concerns fundamental graphs of binary matroids. More specifically, the bipartite graph $G$ of the next lemma is called a \emph{fundamental graph} of $M$.
\begin{lemma}\label{lem:matroid_bipartite}
Let $M$ be a binary matroid. Then $\mathcal{Z}_{2}(M)$ is isomorphic to $\mathcal{Z}_{2}(G)$ for some bipartite graph $G$. Moreover, if $\mathcal{Z}_{2}(M)$ is isomorphic to $\mathcal{Z}_{2}(G)$ for some graph $G$, then $G$ is bipartite.
\end{lemma}
\begin{proof}
If $(I \ A)$ is a representation of $M$, then $(A^T \ I)$ is a representation of $M^*$. Hence $\mathcal{Z}_{2}(M)$ is sheltered by
\[
\begin{pmatrix}
I & A & 0 & 0 \\
0 & 0 & A^T & I
\end{pmatrix}.
\]
By rearranging columns within skew classes, we obtain the matrix
\[
\begin{pmatrix}
I & 0 & 0 & A \\
0 & I & A^T & 0
\end{pmatrix},
\]
which represents $\mathcal{Z}_{2}(G)$ for some graph $G$ with
\[
\begin{pmatrix}
0 & A \\
A^T & 0
\end{pmatrix}
\]
as its adjacency matrix. We observe that $G$ is bipartite.

Finally, assume $\mathcal{Z}_{2}(M)$ is isomorphic to $\mathcal{Z}_{2}(G)$ for some graph $G$. We have that $\mathcal{Z}_{2}(G)$ is represented by
\[
\bordermatrix{~ & \phi_G(V(G)) & \chi_G(V(G)) \cr & I & A(G)}.
\]
Let $f$ be an isomorphism from $\mathcal{Z}_{2}(M)$ to $\mathcal{Z}_{2}(G)$.
Let $T_i = \{(e,i) \mid e \in E(M)\}$ for $i \in \{1,2\}$. Assume that the principal submatrix $A(G)[\chi_G^{-1}(f(T_i))]$ of $A(G)$ induced by $\chi_G^{-1}(f(T_i))$ for some $i \in \{1,2\}$ has a nonzero entry $a_{u,v}$. Since $\phi_G(V(G))$ is a basis of $M[IAS(G)]$, there is a circuit $C \subseteq \phi_G(V(G)) \cup \{\chi_G(v)\}$ with $\chi_G(v) \in C$. Note that $\chi_G(v) \in C \cap f(T_i) \neq \emptyset$ and $\phi_G(u) \in C \cap f(T_j) \neq \emptyset$ with $\{i,j\} = \{1,2\}$. If $u \neq v$, then $C$ is a subtransversal and so $C$ is a circuit of $\mathcal{Z}_{2}(G)$. Consequently, $C' = f^{-1}(C)$ is a circuit of $\mathcal{Z}_{2}(M)$ with $C \cap T_1 \neq \emptyset$ and $C \cap T_2 \neq \emptyset$. This contradicts the fact that $\mathcal{Z}_{2}(M)$ is sheltered by the direct sum of $\varphi_1(M^*)$ and $\varphi_2(M)$. Finally, if $u = v$, then $\mathcal{Z}_{2}(G)$ is not tight by \cite[Proposition 17]{BT1}, contradicting the fact that $\mathcal{Z}_{2}(G)$ is isomorphic to the tight 2-matroid $\mathcal{Z}_{2}(M)$.
Thus the principal submatrices $A(G)[\chi_G^{-1}(f(T_1))]$ and $A(G)[\chi_G^{-1}(f(T_2))]$ are both zero matrices and so $G$ is bipartite.
\end{proof}

We now recall that planar matroids correspond to fundamental graphs that are circle graphs.
\begin{proposition}[Proposition~6 of \cite{F}] \label{prop:Fraysseix}
If $G$ is a bipartite circle graph, then $\mathcal{Z}_{2}(G)$ is isomorphic to $\mathcal{Z}_{2}(M)$ for some planar matroid $M$.

If $M$ is a planar matroid, then $\mathcal{Z}_{2}(M)$ is isomorphic to $\mathcal{Z}_{2}(G)$ for some bipartite circle graph $G$.
\end{proposition}

We call a looped simple graph $G$ a \emph{looped circle graph} if $G$ is obtained from a circle graph by possibly adding some loops.
\begin{proposition} [Corollary~43 in \cite{BT1}] \label{prop:lc_3m}
Let $G$ be a looped simple graph. For every transversal $T$ of $\mathcal{Z}_{3}(G)$, $\mathcal{Z}_{3}(G)-T$ is isomorphic to $\mathcal{Z}_{2}(G')$ for some looped simple graph $G'$ locally equivalent to $G$.
\end{proposition}
\begin{proof}
Corollary~43 in \cite{BT1} shows that $\mathcal{Z}_{3}(G)-T$ is isomorphic to $\mathcal{Z}_{2}(G')$ for some looped simple graph $G'$. The proof of Corollary~43 in \cite{BT1} shows that $G'$ is actually locally equivalent to $G$.
\end{proof}

The main result of this subsection is the following.
\begin{theorem}
Let $M$ be a binary matroid. Then $M$ is planar if and only if $\mathcal{Z}_{3}(M)$ is regular.
\end{theorem}
\begin{proof}
If $M$ is planar, then by Proposition~\ref{prop:Fraysseix} we have that $\mathcal{Z}_2(M)$ is isomorphic to $\mathcal{Z}_{2}(G)$ for some circle graph $G$. By Proposition~\ref{prop:unique_tight_k+1_mm}, $\mathcal{Z}_{3}(M)$ is isomorphic to $\mathcal{Z}_{3}(G)$. By Theorem~\ref{thm:circle_graph_mm_reg}, $\mathcal{Z}_{3}(M)$ is regular.

Conversely, let $\mathcal{Z}_{3}(M)$ be regular. By Theorem~\ref{thm:circle_graph_mm_reg}, $\mathcal{Z}_{3}(M)$ is isomorphic to $\mathcal{Z}_{3}(G)$ for some circle graph $G$. Thus $\mathcal{Z}_{2}(M)$ is isomorphic to $\mathcal{Z}_{3}(G) - T$ for some transversal $T$ of $\mathcal{Z}_{3}(G)$. By Proposition~\ref{prop:lc_3m}, $\mathcal{Z}_{2}(M)$ is isomorphic to $\mathcal{Z}_{2}(G')$ for some looped simple graph $G'$ locally equivalent to $G$. Since $G'$ is locally equivalent to circle graph $G$, $G'$ is a looped circle graph (we use here that circle graphs are closed under simple local complement by, e.g., Theorem~\ref{Bthm}). By Lemma~\ref{lem:matroid_bipartite}, $G'$ is bipartite. Thus $G'$ does not have any loops. Hence $G'$ is a bipartite circle graph. By Proposition~\ref{prop:Fraysseix}, $\mathcal{Z}_{2}(G')$ is isomorphic to $\mathcal{Z}_{2}(M')$ for some planar matroid $M'$. Since $\mathcal{Z}_{2}(M')$ and $\mathcal{Z}_{2}(M)$ are isomorphic, $M$ is planar too.
\end{proof}

\end{document}